\documentclass{amsart}[12pt]

\usepackage{yfonts} %
\usepackage{amssymb} %
\usepackage{amsthm}
\usepackage{array}
\usepackage{booktabs}%
\usepackage{hhline}%
\usepackage{xy} %
\usepackage{epsfig}%
\usepackage{color}%
\usepackage{upgreek}
\usepackage[english]{babel}
\usepackage{epigraph}%
\usepackage{fancybox}%
\usepackage{shadow}
\usepackage{afterpage}
\usepackage{mathrsfs}
\usepackage{enumitem}
\usepackage{subcaption}
\usepackage{graphicx}
\usepackage{type1cm}
\usepackage{eso-pic}
\usepackage{color}
\usepackage[foot]{amsaddr}

\newtheorem{theorem}{Theorem}
\newtheorem{lemma}{Lemma}
\newtheorem{proposition}{Proposition}
\theoremstyle{definition}
\newtheorem{definition}{Definition}
\newtheorem{remark}{Remark}

\theoremstyle{plain}
\newtheorem{corollary}{Corollary}

\newtheorem*{dCW}{Theorem}

\newtheorem*{alexander}{Theorem}

\newcommand{\vt}{\vspace{.1cm}}

\newcommand{\R}{\mathbb{R} }
\newcommand{\q}{\mathbb{Q} }

\newcommand{\N}{\mathbb{N} }
\newcommand{\h}{\mathbb{H} }
\newcommand{\s}{\mathbb{S}}
\newcommand{\ha}{\mathscr{H} }
\newcommand{\har}{\ha^n\times\R}
\newcommand{\g}{{\rm grad}\,}

\renewcommand{\rho}{\varrho}
\renewcommand{\Lambda}{\varLambda}
\renewcommand{\Omega}{\varOmega}
\renewcommand{\theta}{\varTheta}
\usepackage{amsmath}

\newcommand{\overbar}[1]{\mkern 1.5mu\overline{\mkern-1.5mu#1\mkern-1.5mu}\mkern 1.5mu}
\newcommand{\mr}{\overbar M^n\times\R}

\begin{document}

\title[Embeddedness, Convexity,  and Rigidity of Hypersurfaces]
{Embeddedness, Convexity,  and Rigidity of  Hypersurfaces  in Product Spaces}
\author{Ronaldo Freire  de Lima}
\address{Departamento de Matem\'atica -- UFRN \\
Lagoa Nova -- 59.072-970 -- Brasil.}
\email{ronaldo@ccet.ufrn.br}
\maketitle

\begin{abstract}
We establish the following Hadamard--Stoker type theorem:
Let  $f:M^n\rightarrow\har$ be a complete connected hypersurface
with positive definite second fundamental form,
where \,$\ha^n$ is  a Hadamard manifold.
If the height function of $f$ has a critical point, then it is
an embedding and $M$ is homeomorphic to $\s^n$ or  $\R^n.$
Furthermore, \,$f(M)$\,  bounds a convex set in \,$\har.$\,
In addition, it is  shown that, except for the assumption on convexity,  this result
is valid for hypersurfaces in
\,$\s^n\times\R$\, as well.  We apply these theorems
to show that a compact connected hypersurface in \,$\q_\epsilon^n\times\R$\, ($\epsilon=\pm 1$)
is a rotational sphere, provided it has either
constant mean curvature and positive-definite
second fundamental form or constant sectional curvature greater than $(\epsilon +1)/2.$
We also prove that, for \,$\overbar M=\ha^n \,{\rm or} \,\, \s^n,$\,
any connected proper hypersurface \,$f:M^n\rightarrow\overbar M^n \times\R$\, with positive
semi-definite second fundamental form
and height function with no critical points is
embedded and isometric to \,$\Sigma^{n-1}\times\R,$\, where
\,$\Sigma^{n-1}\subset\overbar M^n$\, is convex and homeomorphic to \,$\s^{n-1}$\,
(for \,$\overbar M^n=\ha^n$\, we assume further that \,$f$\, is cylindrically bounded).
Analogous theorems for hypersurfaces in  warped product spaces \,$\R\times_\rho \ha^n$\, and \,$\R\times_\rho\s^n$\,
are obtained.
In all of these results, the manifold \,$M^n$  is assumed to have dimension \,$n\ge 3.$
\end{abstract}

\section{Introduction}

In what concerns embeddedness and convexity of surfaces in Euclidean space,
one of the most fundamental results is the so-called Hadamard--Stoker Theorem.
It states that a complete and positively curved surface \,$S$\, immersed
in Euclidean space is embedded, bounds an open convex set, and is homeomorphic to
a sphere or a plane.  J. Hadamard \cite{hadamard}  proved it, partially, assuming \,$S$\, compact.
Subsequently, J. Stoker \cite{stoker} established the complete case and showed that
\,$S,$\, if  noncompact,  is  a graph over a planar domain.

In this context, another classical result is the celebrated
Cohn-Vossen Rigidity Theorem, according to which a compact  surface of
positive curvature in Euclidean space is rigid, that is, unique up
to Euclidean rigid motions.
(More generally, an isometric immersion
\,$f:M^n\rightarrow\overbar{M}^{n+p}$\, is called \emph{rigid} if,
for any other isometric immersion \,$g:M^n\rightarrow\overbar{M}^{n+p},$\, there is an ambient isometry
\,$\Phi:\overbar{M}\rightarrow\overbar{M}$\, such that \,$g=\Phi\circ f.$\, If so, \,$f$\, and
\,$g$\, are said to be \emph{congruent}.)

In \cite{sacksteder1, sacksteder2}, R. Sacksteder extended both
the Hadamard--Stoker Theorem and the Cohn-Vossen Rigidity Theorem to nonflat hypersurfaces
\,$f:M^n\rightarrow\R^{n+1}$\, with \emph{nonnegative} sectional curvature. Motivated by these results,
M. do Carmo and F. Warner \cite{docarmo-warner} considered hypersurfaces in spherical and
hyperbolic space forms, obtaining then the following theorem.

\begin{dCW}[do Carmo\,--Warner \cite{docarmo-warner}]
Let  $f:M^n\rightarrow\s^{n+1}$ ($n\ge 2$) be a non-totally geodesic hypersurface,
where $M$ is a compact,  connected, and orientable
Riemannian manifold with sectional curvature $K\ge 1.$
Then, the following hold:
\vt
\begin{itemize}[parsep=1ex]
   \item[{\rm a)}] $f$ is an embedding, and  $M$ is homeomorphic  to \,$\s^n.$

   \item[{\rm b)}] $f(M)$  bounds a closed convex set contained in an open hemisphere of \,$\s^{n+1}.$

   \item[{\rm c)}] $f$ is rigid.
\end{itemize}
\vt
\noindent
Moreover, the assertion {\rm (a)} and the convexity property in {\rm (b)}
still hold if one replaces the sphere $\s^{n+1}$ by the hyperbolic space  \,$\h^{n+1},$ and assume that \,$K\ge -1.$
\end{dCW}

We add that the rigidity of compact  hypersurfaces \,$f:M^n\rightarrow\h^{n+1}$\,
with sectional curvature \,$K\ge -1$\, was conjectured by do Carmo and Warner and settled affirmatively
by the author and  R. L. de Andrade in \cite{andrade-delima}.

The conditions on the sectional curvature \,$K$\, of \,$M$\, in each case
of do Carmo\,--Warner Theorem,  spherical and
hyperbolic, can be unified by stating that
the second fundamental form of the hypersurface is semi-definite, that is, its
\emph{extrinsic curvature} is nonnegative (see Section \ref{subsec-hypersurfaces}).
Under the stronger condition of \emph{positive} semi-definiteness of the second fundamental form
(see \cite[Remark 2.4-(b)]{docarmo-warner}),  S. Alexander established the
following Hadamard--Stoker type theorem for compact hypersurfaces in Hadamard manifolds.

\begin{alexander}[Alexander \cite{alexander}]
Let $f:M^n\rightarrow\mathscr{H}^{n+1}$ ($n\ge 2$) be a compact, connected, and oriented hypersurface
in a Hadamard manifold \,$\ha^{n+1}.$\, If the second fundamental form of \,$f$ is positive semi-definite, then
$f$ is an embedding, $M$ is homeomorphic to $\s^n,$ and $f(M)$ bounds an open convex set in $\ha^{n+1}.$
\end{alexander}

Recently, some authors (see, for example, \cite{esp-gal-rosen, esp-oliv, esp-rosen,  oliveira-schweitzer}) have
extended the  Hadamard--Stoker Theorem 
to the context of hypersurfaces in products \,$\overbar{M}\times\R,$\, giving particular attention to
the cases where \,$\overbar M$\, is one of the non flat simply connected space forms.
Considering  all these results, a natural question (raised in \cite{oliveira-schweitzer}) is whether there exist
Hadamard--Stoker type theorems for hypersurfaces in
\,$\ha^n\times\R,$\, where \,$\ha$\, (here and elsewhere) denotes a
general Hadamard manifold. In this paper, we give it an affirmative answer. 

We shall also focus on rigidity of hypersurfaces in
\,$\har$\, and \,$\s^n\times\R.$\, This, however, turns out to be a
delicate matter, since these spaces have
nonconstant sectional curvature.
For instance, in contrast to the behavior of hypersurfaces
in space forms,  two hypersurfaces
of either  \,$\h^n\times\R$\, or \,$\s^n\times\R$\,
with equal shape operators   are not necessarily congruent. As proved by
B. Daniel \cite{daniel}, in order to have congruence in this case, one has to ensure further that
the  height  and  angle functions of the hypersurfaces coincide
(see Section \ref{subsec-hypersurfaces} for definitions).

Taking the above considerations into account, we will investigate the
rigidity of a given hypersurface \,$f:M^n\rightarrow\overbar{M}^{n+1}$\,
in the more restricted class  \,$\mathscr{C}_{\rm ext}(f)$\, of
hypersurfaces \,$g:M^n\rightarrow\overbar{M}^{n+1}$\, whose extrinsic curvature coincides with the extrinsic curvature
of \,$f$\, everywhere on \,$M$\, (cf. \cite[Theorem A]{rosenberg-tribuzy}).
More precisely,  we will say that \,$f$\, is \emph{rigid in} \,$\mathscr{C}_{\rm ext}(f)$\, if, for any
hypersurface \,$g\in\mathscr{C}_{\rm ext}(f),$\,  there exists an isometry
\,$\Phi:\overbar M\rightarrow\overbar M$\, such that \,$g=\Phi\circ f.$\,
 We point out that, when \,$\overbar{M}^{n+1}$\, has constant sectional curvature,
 the concepts of rigidity and rigidity in \,$\mathscr{C}_{\rm ext}(.)$\, are the same.

Our first result, as stated below,  includes a Hadamard--Stoker type theorem
for hypersurfaces in \,$\ha^n\times\R,$\, 
and also a rigidity  theorem
for hypersurfaces in \,$\h^n\times\R$. 

\begin{theorem} \label{th-th1}
Let  $f:M^n\rightarrow{\mathscr{H}}^n\times\R$ ($n\ge 3$) be a
complete connected oriented hypersurface
with positive definite second fundamental form. If the  height function  of  \,$f$
has a critical point, then the following statements hold:

\begin{itemize}[parsep=1ex]
   \item[{\rm a)}] $f$ is a proper embedding, and \,$M$ is either homeomorphic  to $\s^n$ or \,$\R^n.$\, In the latter case,
   \,$f(M)$\, is an unbounded geodesic graph over an open set of a horizontal section of \,$\har.$

   \item[{\rm b)}] $f(M)$ is  the boundary of a convex set in \,${\mathscr{H}}^n\times\R.$

   \item[{\rm c)}] $f$ is rigid in the class \,$\mathscr{C}_{\rm ext}(f)$ when
   \,${\mathscr{H}}^n$ is the hyperbolic space \,$\h^n.$

 \end{itemize}
\end{theorem}

Some of the assertions of Theorem \ref{th-th1} remain true after replacing
the Hadamard manifold \,$\ha^n$\, by the unit sphere \,$\s^n.$\, More precisely, we have the following result.
\begin{theorem} \label{th-th2}
Let  $f:M^n\rightarrow{\s}^n\times\R$ ($n\ge 3$) be a complete connected oriented hypersurface
with positive definite second fundamental form. If the  height function  of $f$
has a critical point, then the following statements hold:
\begin{itemize}[parsep=1ex]
   \item[{\rm a)}] $f$\, is an embedding and  \,$M$\, is  homeomorphic to \,$\s^n$ or \,$\R^n.$ In the latter case,
   the height function of $f$ is unbounded.

   \item[{\rm b)}] $f$ is rigid in $\mathscr{C}_{\rm ext}(f).$
 \end{itemize}
\end{theorem}

Let  \,$\q_\epsilon^n$\, be  \,$\h^n$ ($\epsilon =-1$) or \,$\s^n$ ($\epsilon=1$).
As a  consequence of Theorems \ref{th-th1} and \ref{th-th2},
we have the following extensions of two classical results on surfaces in $\R^3.$

\begin{corollary}[Jellett--Liebmann type theorem] \label{cor-jellett-liebmann}
For $n\ge 3$\, and $\epsilon\in\{-1,1\},$  any compact connected constant mean curvature hypersurface
$f:M^n\rightarrow\q_\epsilon^{n}\times\R$ with positive extrinsic curvature
is congruent to an embedded rotational sphere. 
\end{corollary}

\begin{corollary}[Hilbert--Liebmann type theorem]  \label{cor-hilbert-liebmann}
Let \,$M^n_c$\, be a complete, connected and
orientable $n(\ge 3)$-dimensional Riemannian manifold with constant sectional curvature \,$c.$\,
Given an isometric immersion
\,$f:M_c^n\rightarrow\q_\epsilon^n\times\R,\, \,\epsilon\in\{-1,1\},$\,
assume that \,$c>(1+\epsilon)/2.$\,  Then, \,$f$\, is congruent
to an embedded rotational sphere.
\end{corollary}

Regarding Corollary \ref{cor-hilbert-liebmann},
we should mention that a more general result
was obtained by  Manfio and  Tojeiro in \cite{manfio-tojeiro},
where they classify the  hypersurfaces of constant sectional curvature in \,$\q_\epsilon^n\times\R,$\, \,$n\ge 3.$

Next, we consider the dual case of the above theorems
in which the height function of the hypersurface \,$f$\, has
no critical points. The condition on the second fundamental form is weakened by
assuming it positive semi-definite instead of positive definite. On the other hand,
\,$f$\, is assumed to be a proper immersion, instead of complete.
The precise statement is as follows.

\begin{theorem}  \label{th-nocriticalpoints}
Let $f:M^n\rightarrow\overbar M^n\times\R$ $(n\ge 3)$  be a proper
connected  orientable hypersurface with positive semi-definite second
fundamental form, where $\overbar M^n$ is either a Hadamard manifold
$\ha^n$ or the sphere \,$\s^n.$  Assume that:
\begin{itemize}[parsep=1ex]
  \item[{\rm a)}] The height function of \,$f$ has no critical points.
  \item[{\rm b)}] $f$ is cylindrically bounded if \,\,$\overbar M=\ha^n,$\, that is, there exists a
  closed geodesic ball $B\subset\ha^n$ such that $f(M)\subset B\times\R$.
\end{itemize}
Then, $f$ is an embedding, and
$f(M)=\Sigma\times\R,$ where $\Sigma\subset\overbar M^n\times\{0\}$ is a submanifold
homeomorphic to \,$\s^{n-1}$ which bounds  an open  convex set in \,$\overbar M^n\times\{0\}.$
\end{theorem}

It should be observed  that the assumption (b) in the above theorem is necessary (cf. Remark \ref{rem-neessary}).
It is also worth mentioning that,  in the proof of  the spherical case \,$\overbar M=\s^n,$\, we apply
the celebrated Soul Theorem, due to G. Perelman.


In \cite{currier}, R. Currier  obtained a Hadamard--Stoker type theorem
for complete connected hypersurfaces
\,$f:M^n\rightarrow\h^{n+1}$\, in hyperbolic space.
He proved  that, for such an \,$f,$\,
\,$M$\,  is either homeomorphic to \,$\s^n$\, or  \,$\R^n,$\,  provided
all eigenvalues of its shape operator, at any point,  are at least \,$1.$\,
As shown by the cylinders of \,$\h^{n+1}$\, of constant mean curvature, this hypothesis
on the eigenvalues cannot be replaced by the weaker assumption of positive definiteness of
the second fundamental form.

Currier's result can be viewed  from a more general perspective if we consider \,$\h^{n+1}$\,
as the  warped product \,$\R\times_{e^{t}}\R^n.$\,
In this representation, any \emph{vertical section} \,$\{t\}\times_{e^{t}}\R^n$\, is a
constant mean curvature \,$1$\, horosphere of \,$\h^{n+1}.$\,  Therefore, for a given hypersurface
\,$f:M\rightarrow\R\times_{e^{t}}\R^n,$\, the main hypothesis of Currier's Theorem can
be reinterpreted by saying that the eigenvalues of the shape operator of \,$f$\, at a point \,$x\in M$\,
are all greater than, or equal to, the mean curvature
of the vertical section which contains \,$f(x).$\,

Driven by these considerations, we obtained a Hadamard--Stoker
type theorem for a class of hypersurfaces (which we call $\phi$-convex)
in warped products \,$\R\times_\rho\overbar M^n,$\, where
\,$\overbar M^n$\, is either a Hadamard manifold or the sphere \,$\s^n.$\,
In these spaces, the vertical sections
\,$\{t\}\times_{\rho}\overbar M^n$\, are totally umbilical and, if properly oriented,  have constant mean curvature
\,$|\phi(t)|,$\, where \,$\phi(t)=\rho'(t)/\rho(t).$\, 
In this setting, we  say that a hypersurface in \,$\R\times_\rho\overbar M^n$\, is
$\phi$-\emph{convex} if, at any point,  all the eigenvalues of its shape operator are bounded
bellow by \,$|\phi|$\, (see Section \ref{sec-th2}). The result, then,  reads as follows.

\begin{theorem}  \label{th-warp}
Let $\R\times_\varrho\overbar{M}^n$ ($n\ge 3$)  be a warped product, where $\overbar{M}^n$
is either a Hadamard manifold or the unit sphere $\s^n.$ Consider a connected, complete, and  oriented
strictly $\phi$-convex hypersurface $f:M^n\rightarrow\R\times_\varrho\overbar{M}^n,$
and assume that its height function  has a critical point.
Then, $f$ is a proper embedding and  \,$M$ is either homeomorphic  to \,$\s^n$ or \,$\R^n.$
In the latter case, if \,$\overbar M=\s^n,$\, the height function of $f$ is unbounded.
\end{theorem}

We remark that, when the warping function \,$\rho$\, is constant, strict $\phi$-convexity is equivalent
to definiteness of the second fundamental form. In this manner,
we can say that Theorem \ref{th-warp} is an extension of Theorems \ref{th-th1}-(a) and
\ref{th-th2}-(a). Also, in the case  \,$\overbar M=\ha^n$\, of Theorem \ref{th-warp},
if \,$\rho$\, is convex, then
\,$f$\, behaves as in Theorem \ref{th-th1}. Namely, \,$f(M)$\, bounds a convex set and, when \,$M$\, is noncompact,
is an unbounded geodesic graph (see Remark \ref{rem-convex}).
By the same token, our final result extends the spherical case of Theorem \ref{th-nocriticalpoints} to
proper $\phi$-convex hypersurfaces in \,$\R\times_\varrho\s^n$\, whose height function has no critical points.

\begin{theorem}  \label{th-warpnocriticalpoints}
Assume that there exists a proper,
connected, and  oriented $\phi$-convex hypersurface
$f:M^n\rightarrow \R\times_\varrho\s^n$ $(n\ge 3)$  whose height function
has no critical points.  Then, \,$f$\, is an embedding with unbounded height function (above and below),
and \,$M$ is homeomorphic to the product \,$\R\times\s^{n-1}.$
Assuming, in addition, that $M$ has nonnegative sectional curvature,
the following  hold:
\begin{itemize}[parsep=1ex]

\item ${\rm L}(\varrho)\le 0$ on \,$\R,$\, where \,${\rm L}(\varrho):=(\varrho')^2-\varrho\varrho''.$

\item ${\rm L}(\rho)= 0$ on \,$\R$ if and only if $\rho$ is constant.
\end{itemize}
\end{theorem}

The paper is organized as follows. In Section \ref{sec-preliminaries}, we introduce some notation and quote  some results
which will be used afterwards. In Section \ref{sec-th1}, we prove Theorems \ref{th-th1}--\ref{th-nocriticalpoints}
and Corollaries \ref{cor-jellett-liebmann} and \ref{cor-hilbert-liebmann}
as well. Finally, in Section \ref{sec-th2}, after providing some
background on hypersurfaces in warped products, we prove Theorems \ref{th-warp} and \ref{th-warpnocriticalpoints}.

\section{Preliminaries} \label{sec-preliminaries}

Throughout this paper,
all manifolds  are assumed to be  \,$C^\infty.$\, For a given manifold \,$M,$\, we will
write \,$TM$\, for its tangent bundle.
The simply connected space form of constant sectional  curvature \,$\epsilon\in\{-1,0,1\}$\, and dimension \,$n\ge 2$\,
will be denoted by \,$Q^n_\epsilon,$\, so \,$Q^n_{-1}$\, is the hyperbolic space \,$\h^n,$\,
\,$Q^n_0$\, is the Euclidean space \,$\R^n,$\, and \,$Q^n_1$\, is the unit sphere
\,$\s^n.$\,

Recall that a complete simply connected Riemannian manifold with
non-positive sectional curvature is called a \emph{Hadamard manifold}.
Any Hadamard manifold $\ha^n$ is  diffeomorphic to $\R^n$ through  the exponential map. So, given
\,$p, q\in\ha^n,$\, there exists a unique geodesic \,$\gamma_{pq}$\, joining these two points. A set
\,$\Lambda\subset\ha^n$\, is said to be \emph{convex} if \,$\gamma_{pq}\subset\Lambda$\, whenever
\,$p,q\in\Lambda.$

Given an $n(\ge 2)$-dimensional  Riemannian manifold \,$\overbar M^n,$\, consider
the product \,$\overbar M^n\times\R$\, endowed with the standard Riemannian product
metric. For a tangent vector field \,$Z$\, in  \,$T(\overbar M\times\R)=T\overbar M\oplus T\R,$\, we will write
\[Z=Z_h+Z_v\,, \,\,\, Z_h\in T\overbar M, \,\, Z_v\in T\R,\]
and call \,$Z_h$\, and \,$Z_v$\, the \emph{horizontal component} and the \emph{vertical component} of
\,$Z,$\, respectively.
The projections  of \,$\overbar M^n\times\R$\, onto its first and second factors will be denoted by
\,$\pi_{\scriptscriptstyle\overbar M}$\, and \,$\pi_{\scriptscriptstyle\R},$\, respectively, being \,$\pi_{\scriptscriptstyle\R}$\,
called the \emph{height function}
of \,$\overbar M^n\times\R.$\, The  gradient of \,$\pi_{\scriptscriptstyle\R},$\, which is a parallel field in \,$\overbar M^n\times\R,$\,
will be denoted by \,$\partial_t.$ 

Given \,$t\in\R,$\, the submanifold
\[\overbar M_t:=\overbar M^n\times\{t\}\subset\overbar M^n\times\R\]
will be called the \emph{horizontal section} of \,$\overbar M^n\times\R$\, at level \,$t.$\,
It is easily seen that horizontal sections  are totally geodesic submanifolds of
\,$\overbar M^n\times\R,$\, and that each of them is  isometric to \,$\overbar M^n.$\, For this reason,
we identify the Riemannian connection of
any horizontal section \,$\overbar M_t$\,  with that of \,$\overbar M$\, and denote it by \,$\overbar\nabla.$\,
Geodesics of \,$\overbar M^n\times\R$\, contained in a horizontal section will be called \emph{horizontal}, whereas the
ones tangent to \,$\partial_t$\, will be called \emph{vertical}.

\subsection{Hypersurfaces in product spaces}  \label{subsec-hypersurfaces}
Given an oriented hypersurface
\[f:M^n\rightarrow\mr,\]
we will denote its unit normal field by
$N,$\, its second fundamental form by \,$\alpha$\,, and its shape operator by $A.$\, So,
one has the equalities
\[
\langle\alpha(X,Y),N\rangle=\langle AX,Y\rangle =-\langle\widetilde\nabla_XN,Y\rangle=\langle\widetilde\nabla_XY,N\rangle  \,\,\,\forall X, Y\in TM,
\]
where \,$\langle \,,\, \rangle$\, and  \,$\widetilde{\nabla}$\, stand for the Riemannian metric and
Levi-Civita connection of \,$\mr,$\, respectively.
The  \emph{height function} \,$\xi$\, and the \emph{angle function} \,$\theta$\, of \,$f$\, are defined by the following identities:
\[
\xi(x)=\pi_{\scriptscriptstyle\R}\circ f(x) \quad\text{and}\quad \theta(x)=\langle N(x),\partial_t\rangle, \,\, x\in M.
\]

We shall denote the gradient field and the Hessian form of a function \,$\zeta$\, on \,$M$\,  by
\,$\g\zeta$\, and \,${\rm Hess}\,\zeta,$\, respectively. In particular,
\begin{equation} \label{eq-gradxi}
\g\xi=\partial_t-\theta N.
\end{equation}
This last equality then yields
\[
x\in M \,\,\, \text{is a critical point of} \,\,\, \xi \,\, \Leftrightarrow \,\, N(x)=\pm \partial_t \,\, \Leftrightarrow \,\, \theta(x)=\pm 1\,.
\]
From \eqref{eq-gradxi}, we also have that \,$\widetilde{\nabla}_{X}\,\g\xi=-\theta\widetilde{\nabla}_XN-X(\theta)N.$\, Consequently,
\begin{equation}\label{eq-laplacian}
{\rm Hess}\,\xi(X,Y)=\theta\langle\alpha(X,Y),N\rangle \,\,\, \forall X, Y\in TM.
\end{equation}

Given an open set \,$\Omega\subset M$\,  without critical points of \,$\xi,$\,
a \emph{trajectory} of \,$\g\xi$\, in \,$\Omega$\, is, by definition,
a curve \,$\varphi:I\subset\R\rightarrow\Omega$\, which satisfies
\[
\varphi'(s)=\g\xi(\varphi(s)) \,\, \forall s\in I.
\]
It can be easily proved that, whenever \,$M$\, is complete,
one has \,$I=(-\infty, +\infty).$\, Moreover,
if the closure of \,$\Omega$\,  contains a unique critical point \,$x_0$\,
of \,$\xi,$\, then either
\[
\lim_{s\rightarrow-\infty}\varphi(s)=x_0 \quad\text{or}\quad \lim_{s\rightarrow+\infty}\varphi(s)=x_0\,
\]
according as whether \,$x_0$\, is a local minimum or a local maximum, respectively.
In the first case, one says that the trajectory \,$\varphi$\, is \emph{issuing} from \,$x_0$\,, and, in the second,
that \,$\varphi$\, is \emph{going into} \,$x_0$\, (see \cite{docarmo-lima} for details and proofs).

Concerning the gradient of  \,$\theta$\, on \,$M,$\,
for all \,$X\in TM,$\, we have
\[
X(\theta)=X\langle N,\partial_t\rangle=\langle\widetilde{\nabla}_XN,\partial_t\rangle=-\langle AX,\partial_t \rangle=-\langle A\,\g\xi, X \rangle.
\]
Hence, the following equality holds on \,$M:$
\begin{equation}\label{eq-gradtheta}
  \g\theta=-A\,\g\xi.
\end{equation}


\begin{remark}
When required, we will denote the second fundamental form \,$\alpha$\, of a hypersurface \,$f:M\rightarrow\mr$\,
 by \,$\alpha_f$\,.
The same goes for all  other objects
related to \,$f,$\, including its shape operator \,$A=A_f$\,, and its height and angle functions
\,$\xi=\xi_f$\,, and \,$\theta=\theta_f$\,.
\end{remark}

Consider a hypersurface \,$f:M\rightarrow\mr$\, and
assume that a horizontal section \,$\overbar M_t$\, intersects \,$f(M)$\, transversally. In this case,
as is well known, the set
\,$\xi^{-1}(t)\subset M$\, is an $(n-1)$-dimensional submanifold of \,$M.$\,
Given, then, a connected component \,$M_t$\, of \,$\xi^{-1}(t),$\,
we will call the map
\[
f_t:=f|_{M_t}:M_t\rightarrow\overbar M_t
\]
a \emph{horizontal section} of \,$f$\, at level \,$t.$\,
As unit  normal field  for a horizontal section \,$f_t$\,,
we shall choose the normalized horizontal component \,$\eta$\, of \,$N,$\, that is,
\begin{equation} \label{eq-eta}
\eta:=\frac{N_h}{\|N_h\|}=\frac{\,\,N-\theta\partial_t\,\,}{\sqrt{1-\theta^2}}\,\cdot
\end{equation}

Let $u$ be a differentiable (i.e., $C^\infty$) function defined on a domain $\mathcal D\subset\overbar M.$
The \emph{vertical graph} of \,$u$\, in \,$\overbar M\times\R$\, is defined as the set
\[
\Sigma:=\{(p,u(p))\in \overbar M\times\R\,;\, p\in\mathcal D\}.
\]

It is easily checked that \,$\Sigma$\, is a hypersurface
of \,$\overbar M\times\R$\, (seen as a submanifold). Denoting
by \,$\nabla u$\, the gradient of \,$u$\, in \,$\overbar M$\, and by
\,$\|\nabla u\|$\, its norm, we have that
\begin{equation}\label{eq-normaltograph}
N=\frac{-\nabla u+\partial_t}{\sqrt{1+\|\nabla u\|^2}}
\end{equation}
is clearly a unit normal field on \,$\Sigma$\,
(by abuse of notation, we are writing \,$\nabla u$\, instead of \,$\nabla u\circ\pi_{\scriptscriptstyle\overbar M}$). In particular,
the angle function of \,$\Sigma$\, is
\begin{equation}\label{eq-thetaandu}
\theta=\frac{1}{\sqrt{1+\|\nabla u\|^2}}\cdot
\end{equation}

We shall denote by \,$\Sigma_t$\, the level set of \,$u$\, at \,$t\in u(\mathcal D)\subset\R,$\, that is,
\,$\Sigma_t:=u^{-1}(t).$\, It follows from \eqref{eq-gradxi} and \eqref{eq-normaltograph} that the horizontal component of
\,$\g\xi$\, on \,$\Sigma$\, is parallel to \,$\nabla u,$\, which implies that the projection
\,$\gamma=\pi_{\scriptscriptstyle\overbar M}\circ\varphi$\, of any trajectory
\,$\varphi$\, of \,$\g\xi$\,  to \,$\overbar M$\, is tangent to \,$\nabla u.$\, Thus, such a
\,$\gamma$\,  is  necessarily orthogonal to all  level sets \,$\Sigma_t=u^{-1}(t).$

Let us consider now a general hypersurface
\,$f:M^n\rightarrow\widetilde M^{n+1}.$\,
Recall that the Gauss equation for \,$f$\, is
\begin{equation}  \label{eq-gauss}
\langle R(X,Y)Z,W\rangle=\langle \widetilde{R}(X,Y)Z,W\rangle+\langle\alpha(X,W),\alpha(Y,Z)\rangle-\langle\alpha(X,Z),\alpha(Y,W)\rangle,
\end{equation}
where \,$R$\, and \,$\widetilde R$\, are the curvature tensors of \,$M^n$\, and \,$\widetilde M^{n+1},$\, respectively.
Denoting by \,$K(X,Y)$\, and \,$\widetilde{K}(X,Y)$\, the corresponding sectional curvatures of the plane generated by orthonormal
vectors \,$X, Y\in TM,$\, the Gauss equation becomes
\begin{equation}\label{eq-gauss1}
K(X,Y)=\widetilde K(X,Y)+\langle AX,X\rangle\langle AY,Y\rangle -\langle AX,Y\rangle^2,
\end{equation}
where \,$A$\, is the shape operator of \,$f.$\,

Set \,$\pi_{\scriptscriptstyle XY}$\, for the projection of \,$TM$\, on
\,${\rm span}\,\{X,Y\},$\, and define  the linear operator
\,$A_{\scriptscriptstyle XY}:=\pi_{\scriptscriptstyle XY}A|_{{\rm span}\, \{X,Y\}}\colon{\rm span}\, \{X,Y\}\rightarrow {\rm span}\, \{X,Y\}.$\,
Then, we have
\[
\det A_{\scriptscriptstyle XY}=\langle AX,X\rangle\langle AY,Y\rangle -\langle AX,Y\rangle^2.
\]

Regarding the eigenvalues
\,$\lambda_1\,, \dots ,\lambda_n$\, of the shape operator \,$A,$\,  it is easily shown that
the following assertion holds:
\begin{equation} \label{eq-det}
\lambda_i\ge c\ge 0 \,\,\,\forall i=1,\dots ,n \,\,\, \Rightarrow \,\,\, \det A_{\scriptscriptstyle XY}\ge c^2 \,\,\,\, \forall \{X,Y\} \,\,\text{orthonormal}.
\end{equation}
Furthermore, if the first inequality on the left is strict, so is the one on the right.

The \emph{extrinsic curvature} of \,$f:M^n\rightarrow\widetilde M^{n+1}$\, is defined by
\[K_{\rm ext}(f)(X,Y):=K(X,Y)-\tilde{K}(X,Y), \,\, X,Y\in TM.\]
We shall denote by
\,$\mathscr{C}_{\rm ext}(f)$\, the class of all hypersurfaces \,$g:M^n\rightarrow\overbar{M}^n\times\R$\,
whose extrinsic curvature coincides with that of \,$f,$ that is, those \,$g$\, which satisfy:
\[K_{\rm ext}(f)(X,Y)=K_{\rm ext}(g)(X,Y) \,\, \forall X, Y\in TM.\]

Finally, we remark that, when \,$\widetilde M^{n+1}=Q^n_\epsilon\times\R,$\,
the  equation \eqref{eq-gauss1} takes the form
\begin{equation}\label{eq-daniel}
K(X,Y)=\det A_{\scriptscriptstyle{XY}}+\epsilon(1-\|\pi_{\scriptscriptstyle{XY}}\,\g\xi\|^2)
\end{equation}
(see, e.g.,  \cite{daniel}).

\subsection{Asymptotic rays in \,$\har$}  \label{subsec-asymptoticrays}
Given a Hadamard manifold \,$\ha^n,$\, we have that \,$\har$\, is also a Hadamard manifold. Thus,
we can consider the concept of asymptotic rays in this product space and profit from its properties
(for details and proofs we refer the reader to \cite[Section 9]{bishop-oneill}).

We say that two unit speed geodesic rays \,$\gamma, \sigma:[0,\infty)\rightarrow\har$\, are
\emph{asymptotic} if there is a constant \,$c>0$\, such that
\,${\rm dist}(\gamma(s),\sigma(s))\le c \,\, \forall s\in [0,\infty),$\,
where \,${\rm dist}$\, stands for the distance function on \,$\har.$\,

This concept induces an equivalence relation \,$\sim$\, in the
set of all unit speed geodesic  rays of \,$\har.$\,
The \emph{asymptotic boundary} of \,$\har$\, is  defined as the set of all equivalence classes determined by \,$\sim.$\,
In this setting, we remark the following nice property of geodesic rays:
Given \,$p\in\har,$\, and a geodesic ray \,$\gamma:[0,\infty)\rightarrow\har,$\,  there
exists a unique unit speed geodesic ray \,$\sigma_p$\,   emanating
from \,$p$\, (i.e., \,$\sigma_p(0)=p$)  which is asymptotic to \,$\gamma.$\,
Moreover, the tangent field
\[p\in\har\mapsto\sigma_p'(0)\in T_p(\har)\]
is proven to be continuous (see  \cite[Proposition 9.6]{bishop-oneill}).

The ray \,$\sigma_p$\, can be constructed as follows. Take a sequence \,$s_k\rightarrow +\infty$\, in \,$\R$\, and,
for each \,$k\in\N,$\, consider the geodesic \,$\sigma_k$\, from \,$p$\, to
\,$\gamma(s_k).$\, Then, it is shown that the sequence \,$(\sigma_k)$\, converges to \,$\sigma_p$\, (see \cite[Proposition 9.2]{bishop-oneill}).

Given a geodesic \,$\gamma:\R\rightarrow\har,$\, one has
\[
\frac{d}{ds}\langle\gamma'(s),\partial t\rangle=\langle\widetilde\nabla_{\gamma'}\gamma'(s),\partial t\rangle=0,
\]
that is, the angle between \,$\gamma'(s)$\, and \,$\partial_t$\, is constant along \,$\gamma.$\, In particular,
a complete geodesic of \,$\har$\, is either horizontal or transversal to all  horizontal sections.

By analogy with the idea of vertical graph, 
we shall employ the notion of asymptotic geodesic rays to introduce the following
concept of graph in \,$\har.$

\begin{definition}
Let \,$U\subset\ha_t$\, be a subset of a horizontal section \,$\ha_t$\,.  We say that a set
\,$\mathcal G\subset\har$\, is a \emph{geodesic graph} over \,$U,$\, if there exists a bijection
\,$q\in U\leftrightarrow p=p(q)\in\mathcal G$\, satisfying the following conditions:
\begin{itemize}[parsep=1ex]
  \item For each pair \,$(q,p(q))\in U\times\mathcal G,$\, there is a  geodesic ray \,$\sigma_q$\, emanating from \,$q$\,
  which intersects \,$\mathcal G$\, only at \,$p.$
  \item For all \,$q, q'\in U,$\, \,$\sigma_q$\, is asymptotic to \,$\sigma_{q'}$\,  (Fig. \ref{fig-geodesicgraph}).
\end{itemize}
\end{definition}

\begin{figure}[htbp]
\includegraphics{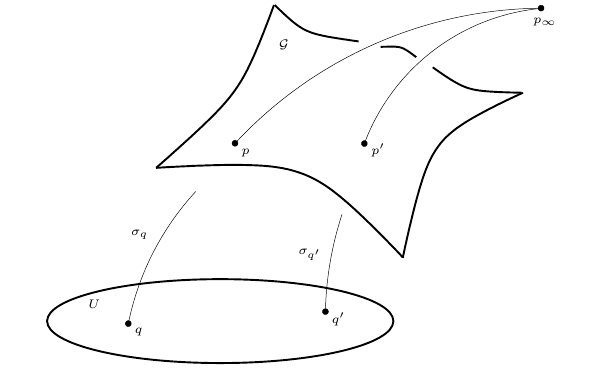}
\caption{A geodesic graph in $\ha^n\times\R$}
\label{fig-geodesicgraph}
\end{figure}

\subsection{Rotational spheres in \,$\q_\epsilon^n\times\R$}  \label{subsec-rotationalspheres}
Concluding this preliminary section, we shall briefly consider rotational hypersurfaces of
\,$\q_\epsilon^n\times\R.$\, Such a  hypersurface is the orbit \,$\Sigma_{\scriptscriptstyle C}$\,  of
a curve \,$C$\, of \,$\q_\epsilon^n\times\R$\, under the action of the group of isometries of
\,$\q_\epsilon^n\times\R$\, which  fix a vertical geodesic \,$\{p\}\times\R, \, p\in\q_\epsilon^n.$\,
The curve \,$C$\, is  called the \emph{profile curve} of \,$\Sigma_{\scriptscriptstyle C}$\,.

By choosing suitable profile curves, one can construct rotational hypersurfaces
in \,$\q_\epsilon^n\times\R$\, with special properties.
The general procedure is analogous to the one for the construction of the well known Delaunay surfaces, that is,
the profile curve \,$C$\, is a solution of a certain differential equation which is obtained from the conditions
imposed on \,$\Sigma_{\scriptscriptstyle C}$\,.

With this approach, nonzero constant mean curvature (CMC, for short) rotational hypersurfaces
in \,$\q_\epsilon^n\times\R$\, were obtained by Hsiang and Hsiang \cite{hsiang}, for \,$\epsilon=-1,$\, and
by R. Pedrosa \cite{pedrosa}, for \,$\epsilon=1.$\,
Furthermore, by applying the Alexandrov reflection technique,  Hsiang and Hsiang were able to prove that
any compact \emph{embedded} CMC hypersurface
of \,$\h^n\times\R$\, is spherical and rotational. 

For compact \emph{embedded} CMC hypersurfaces  \,$f:M^n\rightarrow\s^n\times\R,$\, in general, one can  apply Alexandrov reflection
with respect to horizontal sections \,$S_t:=\s^n\times\{t\}$\, to prove that, for some \,$t^*\in\R,$\,
\,$f(M)$\, is a \emph{bigraph} over its projection \,$\pi(f(M))$\, to \,$S_{t^*}.$\, It means that
\,$S_{t^*}$\, separates \,$f(M)$\, into two symmetric connected components, and each of them is a graph over \,$\pi(f(M)).$\,
If, in addition, there is an open hemisphere \,$\s^n_+$\, of \,$\s^n$\, such that \,$f(M)\subset S_+^n\times\R,$\, then one can
perform Alexandrov reflection on ``hyperplanes'' \,$(\Sigma^{n-1}\cap S_+^n)\times\R,$\, where \,$\Sigma^{n-1}$\, is a totally geodesic
$(n-1)$-sphere of \,$\s^n,$\, and then conclude that \,$f$\, is, in fact,  rotational  
(see \cite[Section 1 -- pg 144]{abresch-rosenberg} and \cite[Section 5]{cheng-rosenberg}).

In \cite{aledo-espinar-galvez},  Aledo, Espinar and Gálvez considered surfaces
of \,$\q_\epsilon^2\times\R$\, with constant sectional curvature.
They showed that, for any given \,$c>(\epsilon+1)/2,$\,
there exists a unique complete surface \,$f:M^2\rightarrow\q_\epsilon^2\times\R$\,
with constant sectional curvature \,$c$\,. Such a surface is necessarily rotational and homeomorphic to
\,$\s^2.$\, As  mentioned in the introduction, an analogous result was obtained by
Manfio and Tojeiro for hypersurfaces \,$f:M^n\rightarrow\q_\epsilon^n\times\R$\, $(n\ge 3$)
as a consequence of their main theorems in \cite{manfio-tojeiro}. In the next section, we shall give
it a new proof (cf. Corollary \ref{cor-hilbert-liebmann}).

\section{Results on Hypersurfaces in \,$\har$\, and \,$\s^n\times\R$}  \label{sec-th1}

For the proof of  assertions (a) and (b) of Theorem \ref{th-th1}, we apply Morse Theory to  show that,
under the given conditions,
the height function of \,$f$\,  has either  one critical point, and then \,$M$\, is homeomorphic to \,$\R^n,$\,
or two critical points, and then  \,$M$\, is  homeomorphic to \,$\s^n.$\, In both cases, \,$f$\, is proven to be a proper
embedding by means of the Alexander Theorem \cite{alexander} (see Introduction).
The convexity property  will be derived from 
a result by  Bishop \cite{bishop}, which states that an embedded
hypersurface in a Riemannian manifold with positive
definite second fundamental form is strictly locally convex.
Then, we apply (a) to show   that, if \,$M$\, is noncompact, then
\,$f(M)$\, is a geodesic graph in \,$\har.$\,
This  part of the proof is based on techniques developed by
Heijenoort \cite{heijenoort}, and  do Carmo and Lima  \cite{docarmo-lima}.
Finally, the rigidity of \,$f$\, in \,$\mathscr{C}_{\rm ext}(f)$\, will be obtained from
B. Daniel's Fundamental Theorem \cite{daniel} for hypersurfaces in \,$Q^n_\epsilon\times\R.$

First, we shall establish the following lemmas.

\begin{lemma} \label{lem-transversal}
Let  $f:M^n\rightarrow\overbar M^n\times\R$ ($n\ge 3$) be an oriented hypersurface
with positive semi-definite (resp. definite) second fundamental form. Then, any
horizontal section $f_t:M_t\rightarrow\overbar M_t$ of \,$f$, if properly oriented,  
has positive semi-definite (resp. definite) second fundamental form.
\end{lemma}

\begin{proof}
Since \,$\overbar M_t$\, is  totally geodesic in
\,$\overbar M^n\times\R,$\, we have that
\[
\overbar\nabla_XX=\widetilde{\nabla}_XX  \,\,\,  \forall X\in T\overbar M_t\,.
\]
Thus, orienting \,$f_t$\, as in \eqref{eq-eta}, for all \,$X\in TM_t$\,, 
we have
\[
\langle\alpha_{f_t}(X,X),\eta\rangle=\langle\overbar\nabla_XX,\eta\rangle = 
\frac{1}{\sqrt{1-\theta^2}}\langle\widetilde{\nabla}_XX,N\rangle
=\frac{1}{\sqrt{1-\theta^2}}\langle\alpha_f(X,X),N\rangle\ge 0.
\]
Hence,  \,$\alpha_{f_t}$\, is positive semi-definite.
If \,$f$\, has positive definite second fundamental form,  the above inequality
is strict, and then \,$f_t$\, has positive definite second fundamental form as well.
\end{proof}

\begin{lemma} \label{lem-rigidity}
Let $f:M^n\rightarrow\q_\epsilon^n\times\R$ $(n\ge 3)$
be an oriented hypersurface whose shape operator $A_f$  has rank at least $3$ everywhere.
Then, if $g:M^n\rightarrow\q_\epsilon^n\times\R$ is a hypersurface in  $\mathscr{C}_{\rm ext}(f)$\,
(see Section \ref{subsec-hypersurfaces}), there exists
a unit normal field $N_g\in TM_g^\perp$ such that the corresponding shape operator $A_g$,  the height function
$\xi_g$, and the angle function $\theta_g$ satisfy the following identities:
\begin{itemize}[parsep=1ex]
  \item $A_f=A_g$.
  \item $\|\g\xi_f\|=\|\g\xi_g\|.$
  \item $\theta_f^2=\theta_g^2.$
\end{itemize}
\end{lemma}

\begin{proof}
Since the sectional curvature determines the tensor curvature, and
the extrinsic curvatures of \,$f$\, and \,$g$\, coincide, it follows from  Gauss equation \eqref{eq-gauss} that
\begin{eqnarray}
\langle\alpha_f(X,W),\alpha_f(Y,Z)\rangle &- & \langle\alpha_f(X,Z),\alpha_f(Y,W)\rangle    \nonumber\\
                                                                                  & = &  \langle\alpha_g(X,W),\alpha_g(Y,Z)\rangle-\langle\alpha_g(X,Z),\alpha_g(Y,W)\rangle \nonumber
\end{eqnarray}
for all \,$X,Y,Z,W\in TM.$\,
Therefore, since the rank of \,$A_f$\, on \,$M$\, is at least \,$3,$\,  Lema 2.1 of \cite{dajczer-rodriguez} applies and gives that
there exists an isometric bundle isomorphism
\,$\mathfrak{B}: TM_f^\perp\rightarrow TM_g^\perp$\,
satisfying \,$\alpha_g=\mathfrak{B}\circ\alpha_f$. In particular,  \,$N_g:=\mathfrak{B}N_f$\, is a unit
normal field to \,$g.$\, Denoting by \,$A_g$\, the shape operator of \,$g$\, with respect to \,$N_g$\,,
for all \,$X, Y\in TM,$\, one has
\[
\langle A_gX,Y\rangle N_g=\alpha_g(X,Y)=\mathfrak{B}\alpha_f(X,Y)=\langle A_fX,Y\rangle\mathfrak{B} N_f=\langle A_fX,Y\rangle N_g\,,
\]
which implies that
\,$A_f=A_g$\, everywhere on \,$M.$\,
Considering now equality \eqref{eq-daniel}, we have that
\,$\|\g\xi_f\|=\|\g\xi_g\|$\, and, from \eqref{eq-gradxi}, that
\,$\theta_f^2=\theta^2_g$\, on \,$M.$
\end{proof}

\begin{proof}[Proof of Theorem \ref{th-th1}]
Let  \,$x_0\in M$\, be a critical point of the height function \,$\xi.$\, We can assume
without loss of generality that \,$x_0$\, is a local minimum, and that
\,$\xi(x_0)=0.$\, Since the second fundamental form \,$\alpha$\, is positive definite,
equality \eqref{eq-laplacian} gives that  \,${\rm Hess}\,\xi$\, is positive definite at \,$x_0$\,,
which implies that \,$x_0$\, is a strict local minimum point of \,$\xi.$

Suppose that  \,$f_t:M_t\rightarrow\ha_t$\, is a horizontal section of \,$f$\, at  level \,$t>0.$\,
Following do Carmo and Lima \cite{docarmo-lima}, we say that \,$f_t$\, (or, equivalently, \,$M_t$)
is a \emph{normal section} (for \,$x_0$)  if the following conditions are satisfied:

\begin{itemize}[parsep=1ex]
  \item $M_t$\, is homeomorphic to \,$\s^{n-1}$\, and bounds an open region \,$\Omega_t\subset M$\,
  which contains only one critical point of \,$\xi;$\, namely, \,$x_0.$

  \item There exists an homeomorphism  \,$\psi:{\rm cl}\, B \rightarrow {\rm cl}\,\Omega_t$\, such
  that \,$\psi(\partial B)=M_t$\,, where \,$B$\, is an open  ball of \,$\R^n$ and
  ${\rm cl}$ denotes closure.
\end{itemize}

When \,$f_t$\, is a normal section,  we say that \,$t$\, is a \emph{normal value}  and  \,$\Omega_t$\, is
a \emph{normal region} for \,$x_0$\,. We then write
\[
{I}:=\{t>0 \,;\, t \,\, \text{is a normal value}\} \quad\text{and}\quad \Omega:=\bigcup\Omega_t\,, \,\, {t\in I}.
\]

It is clear from its definition that
\,$\Omega$\, is a nonempty open set of $M$\, which is  homeomorphic to \,$\R^n.$
Setting \,$\partial\Omega$\, for the boundary of \,$\Omega,$\, we distinguish the following mutually exclusive cases:

\begin{itemize}[parsep=1ex]
  \item[i)] $\Omega=M,$\, i.e., \,$\partial\Omega=\emptyset.$
  \item [ii)]$\Omega\ne M$\, and \,$\partial\Omega$\, contains critical points of \,$\xi.$
  \item [iii)] $\Omega\ne M$\, and \,$\partial\Omega$\, contains no critical points of \,$\xi.$
\end{itemize}

By Lemma \ref{lem-transversal},  each normal section \,$f_t:M_t^{n-1}\rightarrow\ha^n_t$\,
has positive definite second fundamental form. Since we are assuming \,$n\ge 3,$\, it follows
from  Alexander Theorem  that \,$f_t$\, is an embedding and \,$f(M_t)$\,
bounds a compact convex set in \,$\ha_t.$\, In particular, for all \,$t\in I,$\,
\,$f|_{\Omega_t}$\, is a proper embedding. Thus,
\,$f(\Omega_t)$\, separates \,$\ha^n\times[0,t)$\, into two connected components,
where one of them, say \,$\Lambda_t$\,, is bounded.

We claim that  \,$\Lambda_t$\, is convex. To see this,  observe first that the mean
curvature vector of \,$f$\, along \,$\Omega_t$\, points to \,$\Lambda_t$\,,
that is, \,$\Lambda_t$\, is the \emph{mean convex side} of \,$f|_{\Omega_t}$\,.
Since the second fundamental form of
\,$f$\, is positive definite, a theorem by R. Bishop \cite{bishop} gives that \,$f$\, is strictly locally convex, that is,
for each \,$x\in\Omega_t$\,,  there is a neighborhood \,$V\subset T_xM$\,
of the null vector in the tangent space of \,$M$\, at \,$x,$\, such
that \,$\exp_{f(x)} V\cap {\rm cl}\,\Lambda_t=\{f(x)\}.$\,  Here, \,$\exp$\, stands for the exponential map of
the ambient space \,$\har.$\,

\begin{figure}[htbp]
\includegraphics{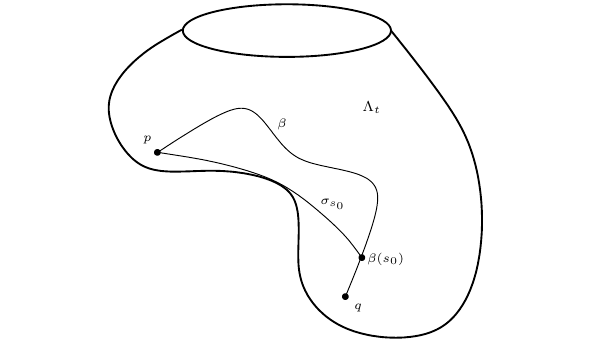}
\caption{Proof that $\Lambda_t$ is convex}
\label{HnXR-fig1}
\end{figure}

Suppose, for the sake of contradiction, that there are points \,$p, q\in\Lambda_t$\, such that
the geodesic of \,$\har$\, which joins them  is not contained in \,$\Lambda_t$\,. Consider, then, a curve
\,$\beta:[0,1]\rightarrow\Lambda_t$\, such that \,$\beta(0)=p$\, and \,$\beta(1)=q.$\,
For each \,$s\in (0,1],$\, let \,$\sigma_s$\, be the geodesic from \,$p$\, to \,$\beta(s).$\, For a small
\,$s,$\, \,$\sigma_s$\, is in \,$\Lambda_t$\,. So, there is \,$s_0\in (0,1]$\, such that
\,$\sigma_{s_0}$\, is tangent to \,$f(M)$\, at some point, which clearly violates the local convexity of
\,$f$\, (Fig. \ref{HnXR-fig1}).  Therefore, \,$\Lambda_t$\, is convex.

Suppose that (i) holds. In this case, 
\,$M=\Omega=\bigcup\Omega_t$\, is homeomorphic to \,$\R^n.$\, Also, by
the above considerations, 
\,$f$\, is a proper embedding, and \,$f(M)$\, is  the boundary of the open convex  set \,$\Lambda:=\bigcup\Lambda_t$\,, \,$t\in I.$

Let us prove that the interval $I$ must be unbounded if \,$\Omega=M.$\,  Assuming otherwise,
we have \,$t^*:=\sup I<\infty.$ Hence, the horizontal section \,$\mathscr{H}_{t^*}$
is disjoint from \,$f(M),$\, for \,$\xi$\,
has no critical points but $x_0$\,. In particular, \,$\mathscr{H}_{t^*}\subset\Lambda.$\,
Consider an arbitrary divergent sequence \,$p_k\in\mathscr{H}_{t^*}, \,\, k\in\N.$\,
For each \,$k\in\N,$\, let \,$\gamma_k:[0, a_k]\rightarrow\Lambda$\, be the unit speed geodesic
of \,$\har$\, from \,$f(x_0)$\, to \,$p_k$\,. Passing to a subsequence, if necessary, we can assume that
\[\gamma_k'(0)\rightarrow Z_0\in T_{f(x_0)}(\har), \,\,\, \|Z_0\|=1.\]

Denote by \,$\gamma:[0,+\infty)\rightarrow\har$\, the unit speed geodesic ray of \,$\har$\, such that
\,$\gamma(0)=f(x_0)$\, and \,$\gamma'(0)=Z_0$\,.
Clearly, each geodesic segment $\gamma_k$ is contained in the closure of the  convex set $\Lambda'\subset\har$
bounded by \,$f(M)$\, and \,$\ha_{t^*},$\, which implies that  \,$\gamma\subset {\rm cl}\,\Lambda'.$\,
Consequently,
\,$\gamma$\, is a  horizontal geodesic ray emanating from \,$f(x_0),$\, i.e.,
it is entirely contained in \,$\ha_0$\, (otherwise, it would
be transversal to \,$\ha_{t^*}$\, and would not be contained in \,${\rm cl}\,\Lambda'$).
However, the only point of \,$\ha_0$\, in \,${\rm cl}\,\Lambda'$\, is \,$f(x_0),$\,  which is a contradiction.
Therefore, if (i) occurs, \,$I$\, is unbounded, and so is the height function of $f.$


\begin{figure}[htbp]
\includegraphics{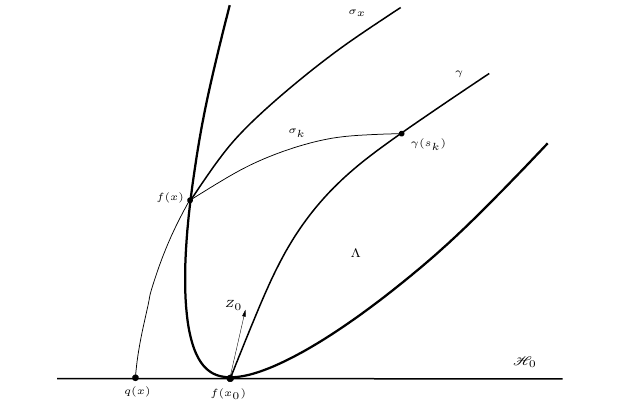}
\caption{Proof that $f(M)$ is a geodesic graph}
\label{HnXR-fig2}
\end{figure}

Still assuming (i), we shall  show that
\,$f(M)$\, is  a geodesic graph over an open connected set $U$ of \,$\ha_0$\,.
For that,  consider an arbitrary divergent sequence \,$p_k\in\Lambda, \,\, k\in\N.$\,
Just as above, construct from that sequence a unit speed geodesic ray
\,$\gamma:[0,+\infty)\rightarrow\har$\, emanating  from \,$f(x_0).$\,
Again, since \,$\Lambda$\, (and so \,${\rm cl}\,\Lambda$) is convex,
\,$\gamma$\,  is  contained in \,${\rm cl}\,\Lambda.$\, Moreover, the local convexity of
\,$f$\, implies that \,$f(x_0)$\, is the only point of \,$\gamma$\, in \,${\rm cl}\,\Lambda-\Lambda.$\,

Given \,$x\in M,$\,
consider  a sequence \,$\gamma(s_k)$\, on \,$\gamma$\, such that \,$s_k\rightarrow +\infty,$\, and let
\,$\sigma_k$\, be the unit speed geodesic from \,$f(x)$\, to \,$\gamma(s_k).$\,
As we know,  this  sequence of geodesics converges to a  geodesic ray
\,$\sigma_x$\, which  is
asymptotic to \,$\gamma.$\, Just as \,$\gamma,$\, \,$\sigma_x$\, is contained in \,${\rm cl}\,\Lambda,$\, and
\,$f(x)$\, is the only point of \,$\sigma_x$\, in \,${\rm cl}\,\Lambda-\Lambda$\, (Fig. \ref{HnXR-fig2}).

Notice that \,$\sigma_x$\, is not a horizontal geodesic, for
the horizontal sections of \,$f$\, are compact. Thus, the complete geodesic that contains \,$\sigma_x$\, is transversal to all horizontal sections of
\,$\har.$\, In particular, the geodesic ray \,$-\sigma_x$\, starting at \,$f(x)$\, in the direction
\,$-\sigma_x'(0)$\,  reaches \,$\ha_0$\, at some point \,$q=q(x),$\, which implies that
\,$f(M)$\, is a  geodesic graph over the set \,$U=\{q(x), \, x\in M\}\subset\ha_0.$\,
Moreover, since the exponential map and  the field
\,$x\in M \mapsto \sigma_x'(0)\in T_{f(x)}(\har)$\, are continuous (see Section \ref{subsec-asymptoticrays}), the
map \,$x\in M\mapsto q(x)\in U$\, is clearly a homeomorphism. In particular,
\,$U$\, is connected and open in  \,$\ha_0$.

Henceforth, we will assume that \,$I$\, is bounded, that is,
\,$t^*=\sup I<\infty.$\, Under this hypothesis, suppose that
(ii) holds and let \,$x_1\in\partial\Omega$\, be a critical point of \,$\xi.$\, Then,
\,$\xi(x_1)=t^*$\, and \,$N(x_1)=\pm \partial_t.$\, However,  \,$N(x_1)= \partial_t$\, would give that \,$x_1$\, is a strict
local minimum for \,$\xi.$\,  In that case, there would exist a neighborhood \,$V$\, of \,$x_1$\, in \,$M$\,
such that \,$\xi|_{V-\{x_1\}}>t^*,$\,
contradicting the fact that \,$x_1$\, is on the boundary of \,$\Omega.$\,
Thus, \,$N(x_1)=- \partial_t$\, and
\,$x_1$\, is a strict local maximum of \,$\xi.$\, In particular,  \,$f(x_1)$\, is  isolated in \,$\ha_{t^*}.$
Consequently, in the occurrence of (ii),
\,$M$\, coincides with the closure of \,$\Omega$\, and is, in particular, compact.
Therefore, \,$\xi$\, is a Morse function on \,$M$\, with only two critical points, which
implies that \,$M$\, is homeomorphic to a sphere (see \cite[Theorem 4.1]{milnor}).
The proofs that \,$f$\, is an embedding and that \,$f(M)$\, bounds a convex set
in \,$\mathscr{H}^n\times\R$\, are the same as in case (i) (these facts also follow from  Alexander Theorem).

Finally, we shall prove that  (iii) is impossible.       
Assume, to the contrary, that  (iii) holds.
In this case, \,$M_{t^*}:=\partial\Omega\subset\xi^{-1}(t^*)$\, is a connected $(n-1)$-dimensional submanifold
of \,$M$\, which arises as  the limit set of \,$M_t$\, as \,$t\rightarrow t^*.$\,
Thus, since \,$f|_{M_t}$\, is a proper embedding for all \,$t\in (0,t^*),$\, the same is true for
\,$f|_{{\rm cl}\,\Omega}:{\rm cl}\,\Omega\rightarrow\har.$\,
Furthermore,  if we set \,$\Lambda=\bigcup\Lambda_t\,, \,\, t\in (0,t^*),$\,
we have that \,${\rm cl}\,\Lambda$\, is convex.

Suppose that \,${\rm cl}\,\Lambda$\, is unbounded in \,$\har.$\, Then, there exists a divergent sequence
\,$p_k\in {\rm cl}\,\Lambda.$\,  
As we did before, from this sequence we  obtain
a geodesic ray \,$\gamma$\, emanating from \,$f(x_0)$\, which is  contained in
\,${\rm cl}\,\Lambda.$\, Hence, it  must be horizontal. Again, this contradicts that the only point
of \,${\rm cl}\,\Lambda$\, on \,$\ha_0$\,  is \,$f(x_0).$\,
Consequently, \,${\rm cl}\,\Lambda$\, is bounded and, therefore, compact.

Since \,$f|_{{\rm cl}\,\Omega}$\, is a proper embedding and
\,${\rm cl}\,\Lambda$\,  is compact, we have that
\,$M_{t^*}=\partial\Omega$\, is compact.
Hence, for a given \,$t\in (0, t^*),$\, the flow of \,$\g\xi$\, from \,$M_t$\,
to \,$M_{t^*}$\, is a homeomorphism (see  \cite[Theorem 3.1]{milnor}).
Then, by following the trajectories of \,$\g\xi$\, through  \,$M_{t^*}$\,, one can arrive at
a normal region \,$M_{t'}$\, for a sufficiently small \,$t'>t^*,$\, which is  a contradiction.
This shows the impossibility of (iii) and  finishes
the proof of  assertions  (a) and  (b) of the theorem.

To prove (c), let us consider a hypersurface \,$g:M^n\rightarrow\h^n\times\R$\,
in \,$\mathscr{C}_{\rm ext}(f).$\,
Since \,$\alpha_f$\, is positive definite, we have that its shape operator has
rank \,$n\ge 3$\, everywhere. So,
from Lemma \ref{lem-rigidity}, with respect to a suitable normal field \,$N_g\in TM_g^\perp,$\,
one has \,$A_f=A_g$\,, \,$\|\g\xi_f\|=\|\g\xi_g\|,$\, and \,$\theta_f^2=\theta_g^2.$\,
In particular,
the set of critical points of \,$\xi_f$\, and \,$\xi_g$\, coincide and, then,
\,$f$\, shares with \,$g$\, all the properties stated in (a) and (b).

Now, set \,$A:=A_f=A_g$\,, let \,$\theta$\, be either
\,$\theta_f$\, or \,$\theta_g$\,, and let
\,$\varphi:\R\rightarrow M$\,
be a trajectory of either \,$\g\xi_f$\, or \,$\g\xi_g$\,.  Then, by \eqref{eq-gradtheta},
\begin{equation}\label{eq-thetadedreasing}
\frac{d}{ds}\theta(\varphi(s))=\langle\g\theta(\varphi(s)),\varphi'(s)\rangle=-\langle A\varphi'(s),\varphi'(s)\rangle <0,
\end{equation}
that is, the angle functions \,$\theta_f$\, and \,$\theta_g$\, are both decreasing along
the trajectories of  \,$\g\xi_f$\, and \,$\g\xi_g$\,, respectively.
Also, differentiating the equality   \,$\theta_f^2=\theta_g^2$\,
and using \eqref{eq-gradtheta}, we easily conclude that
\[\theta_f\,\g\xi_f=\theta_g\,\g\xi_g\,.\]

After a possible reflection about an horizontal section of \,$\h^n\times\R$\, (which is an isometry),
we can assume that \,$x_0\in M$\, is a minimum point of \,$\xi_g$\,. In this case, one has
\,$\theta_f(x_0)=\theta_g(x_0)=1.$\, Since \,$\theta_f=\pm\theta_g$\,, by continuity,
\,$\theta_f=\theta_g$\, in a neighborhood  $V$ of \,$x_0$\,, which gives that
\,$\g\xi_f=\g\xi_g$\, on \,$V.$\,
Hence, on \,$V,$\, the trajectories of \,$\g\xi_f$\, and \,$\g\xi_g$\, coincide.
However,  \,$\theta_f\,$\, and \,$\theta_g$\, are both decreasing along these trajectories. Thus,
the identity \,$\theta_f=\theta_g$\,, and so \,$\g\xi_f=\g\xi_g$\,, extends to all of \,$M.$

It follows that the equalities
\[
A_f=A_g\,, \quad \theta_f=\theta_g \quad\text{and}\quad \g\xi_f=\g\xi_g
\]
hold everywhere in \,$M.$\, Therefore,  by  Daniel Theorem \cite{daniel},
there exists an isometry
\[\Phi:\h^n\times\R\rightarrow\h^n\times\R\]
such that
\,$g=\Phi\circ f.$\, This shows (c) and concludes the proof of the theorem.
\end{proof}

\begin{proof}[Proof of Theorem \ref{th-th2}]
We just sketch the proof, since the argument is similar to
the one in the proof of Theorem \ref{th-th1}.
By Lemma \ref{lem-transversal}, the horizontal sections of \,$f$\, have positive definite second fundamental form.
Considering do Carmo\,--Warner Theorem and  keeping the notation of the proof of Theorem \ref{th-th1},
one has that all normal sections of \,$f$\, are compact and embedded.
Thus, if (i) occurs, \,$M$\, is homeomorphic to \,$\R^n$\, and \,$f$\, is properly embedded. From this last property,
\,$f(M)$\, is not contained in the compact region bounded by two horizontal sections of \,$\s^n\times\R,$ which implies that the
the height function of \,$f$\, is unbounded.
The possibility (ii), analogously,
gives that \,$f$\, is properly embedded and that \,$M$\, is homeomorphic to \,$\s^n.$\,
The possibility (iii) is easily ruled out, for
the horizontal sections of \,$\s^n\times\R$\, are necessarily compact. This proves (a).
Regarding (b), we have just to consider Lemma \ref{lem-rigidity}, and remember that Daniel Theorem \cite{daniel}
is set in \,$\h^n\times\R$\, and \,$\s^n\times\R$\, as well.
\end{proof}


\begin{proof}[Proof of Corollary \ref{cor-jellett-liebmann}]
Since \,$M$\, is compact, the height function of \,$f$\, has a critical point. Therefore, from Theorems \ref{th-th1} and \ref{th-th2}, \,$f$\, is an embedding
and \,$M$\, is homeomorphic to \,$\s^n.$\, Thus, for \,$\epsilon =-1,$\, the main theorems in
\cite{hsiang}  give that \,$f$\, is congruent to a rotational  sphere of positive constant mean curvature.

Let us consider now the case \,$\epsilon=1.$\, In this setting, as discussed in Section \ref{subsec-rotationalspheres},
we can perform Alexandrov reflection with respect to horizontal sections
\,$S_t=\s^n\times\{t\}$\, to conclude that, for some \,$t^*\in\R,$\,
\,$f(M)$\, is a bigraph over  its projection \,$\pi(f(M))$\, to \,$S_{t^*}$\,. Thus, writing
\,$f_{t^*}:M_{t^*}\rightarrow S_{t^*}$\, for the  horizontal section of \,$f$\,  at \,$t^*$\,, we have that
\,$f(M_{t^*})$\, is the boundary of \,$\pi(f(M)).$\,

By Lemma \ref{lem-transversal}, \,$f_{t^*}$\, has positive extrinsic
curvature. In particular, it is non totally geodesic. Thus,
by do Carmo\,--Warner Theorem, \,$f(M_{t^*})$\, is contained in an open hemisphere \,$S_{t^*}^+$\,
of \,$S_{t^*}$\,. So, the same is true for \,$\pi(f(M)),$\,
that is, \,$f(M)\subset S_{t^*}^+\times\R.$\,
As also discussed in Section \ref{subsec-rotationalspheres},
this implies that we can perform Alexandrov reflections on
\,$f(M)$\, and conclude that it
is congruent to a rotational sphere of positive
constant mean curvature.
\end{proof}

\begin{proof}[Proof of Corollary \ref{cor-hilbert-liebmann}]
Since \,$c>(1+\epsilon)/2\ge 0,$\,  by Myers Theorem,
\,$M$\,  is compact. Also, it is easily seen that the
maximum value of the sectional curvature in \,$\q_\epsilon^n\times\R$\, is
\,$(1+\epsilon)/2.$\, This, together with  Gauss equation and  the condition on \,$c,$\,
implies that \,$f$\, has positive definite second fundamental form, if properly oriented.
Thus, Theorems \ref{th-th1} and \ref{th-th2}  apply and give that
\,$f$\, is an isometric embedding of the standard sphere
\,$S_c^n$\, of constant sectional curvature \,$c$\, into \,$\q_\epsilon^n\times\R.$\,
In addition, as seen in the proofs of these theorems, the height function \,$\xi$\, of \,$f$\, has exactly
two critical points; a minimum \,$x_0$\, and a maximum \,$x_1$\,. As before, assume
\,$\xi(x_0)=0$\, and observe that \,$\theta(x_0)=1$\, and \,$\theta(x_1)=-1.$

By \cite[Lemma 3.1]{manfio-tojeiro} (see also \cite[Lemma 5]{lps}),
$\g\xi$\, is an eigenvector of the shape operator \,$A$\,  on \,$M-\{x_0\,,x_1\}.$\, Since
tangent vectors of horizontal sections \,$f_t:M_t\rightarrow\q_\epsilon^n\times\{t\}$\,
are orthogonal to \,$\g\xi,$\, it follows from \eqref{eq-gradtheta} that the angle function
\,$\theta$\, of \,$f$\, is constant along the horizontal sections \,$f_t$\,.

\begin{figure}[htbp]
\includegraphics{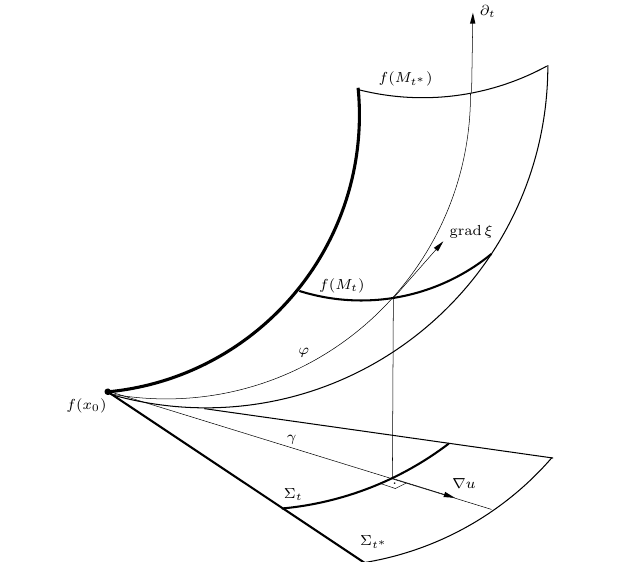}
\caption{A piece of the rotational graph $\Sigma$}
\label{fig-rotational}
\end{figure}

Notice that the trajectories of \,$\g\xi$\, cover \,$M-\{x_0\,,x_1\}$\,, and any of them issues  from \,$x_0$\, and
goes into \,$x_1$\, (cf. Section \ref{subsec-hypersurfaces}). Moreover,
along these trajectories,  \,$\theta$\, decreases from \,$1$\, to \,$-1$\, (see \eqref{eq-thetadedreasing}).
Therefore, for some \,$t^*\in (0,\xi(x_1)),$\, the angle function \,$\theta$\, vanishes on \,$M_{t^*}$\,, and is
positive on  \,$M_{t}$\, for \,$t\in(0, t^*).$\, In particular, the set
$\Sigma$\,  of all points of \,$f(M)$\,
at height less than \,$t^*$\,  is a vertical graph of a differentiable function \,$u$\, on the
projection \,$\mathcal D:=\pi(\Sigma)\subset\q_\epsilon^n.$\, It should also be noticed that
the level hypersurfaces \,$\Sigma_t:=u^{-1}(t)\subset\mathcal D,$\,  \,$t\in(0, t^*),$\, are all topological $(n-1)$-spheres.

Since \,$\theta$\, is constant along the horizontal sections of \,$f,$\, it follows
from  \eqref{eq-thetaandu} that \,$\|\nabla u\|$\, is constant along the level spheres \,$\Sigma_t$\,.
Hence, any trajectory \,$\gamma$\, of \,$\nabla u$\, is actually
a geodesic of \,$\q_\epsilon^n$ (cf. \cite[Lemma 1]{tojeiro}). 
However, as we have discussed before,  such a geodesic \,$\gamma$\, is nothing but the (reparametrized)
projection of a trajectory of $\g\xi$\, to \,$\mathcal D.$\, So, any \,$\gamma$\, is a  geodesic
of \,$\q_\epsilon^n$\, emanating from \,$f(x_0)$\,
and  orthogonal to all level spheres \,$\Sigma_t$  (Fig. \ref{fig-rotational}).

We conclude from this last statement that,
for all \,$t\in (0,t^*],$\, \,$\Sigma_t$\, is a round
geodesic sphere of \,$\q_\epsilon^n$\, centered  at \,$f(x_0).$
Thus,  \,$\mathcal D$\, is the open  ball of
\,$\q_\epsilon^n$\, centered at \,$f(x_0)$\, and bounded by \,$\Sigma_{t^*}$\,, and
the graph \,$\Sigma\subset f(M)$\, is rotational with
axis \,$f(x_0)\times\R$\, and boundary \,$f(M_{t^*}).$\,

An analogous reasoning applied
to \,$\Sigma'=f(M)-{\rm cl}\,\Sigma$\, leads to the conclusion that \,$\Sigma'$\, is rotational with
axis \,$f(x_1)\times\R$\, and boundary \,$f(M_{t^*}).$\, Since \,$\Sigma$\, and \,$\Sigma'$\, are both rotational
and have the geodesic $(n-1)$-sphere \,$f(M_{t^*})$\, as their common boundary,
their axes must coincide. Hence,
\,$f(M)={\rm cl}\,\Sigma\cup{\rm cl}\,\Sigma'$\,
is indeed a rotational sphere, as we wished to prove.
\end{proof}

Recall that, as proved by Cheeger and Gromoll \cite{cheeger-gromoll}, any complete and noncompact  Riemannian manifold
\,$M$\, with nonnegative sectional curvature has a compact submanifold, called the \emph{soul} of \,$M,$\, whose normal
bundle in \,$M$\, is diffeomorphic to \,$M.$\, If \,$M$\, has at least one  point at which all sectional curvatures are positive,
then the soul of \,$M$\, is a single point (and then \,$M$\, is homeomorphic to \,$\R^n$). This fact, conjectured by
Cheeger and Gromoll, was proved by G. Perelman in \cite{perelman}, and called the \emph{Soul Theorem.}
In what follows, we shall give a proof of Theorem \ref{th-nocriticalpoints} in which the Soul Theorem
plays a crucial role. 

\begin{proof}[Proof of Theorem \ref{th-nocriticalpoints}]
Since  \,$B_t:=B\times\{t\}$\,
and \,$S_t:=\s^n\times\{t\}$\, are both compact,
the height function \,$\xi$\, is unbounded above and below on \,$M,$\,
for \,$f$\, is proper  and  \,$\xi$\,  has no critical points.
Thus, setting \,$\overbar M_t$\, for \,$B_t$\, or \,$S_t$\,,
one has \,$M=\bigcup M_t$\,, \,$t\in\R,$\,
where \,$f_t:M_t\subset M\rightarrow\overbar M_t$\, is a family of
compact and connected horizontal sections of \,$f.$

By Lemma \ref{lem-transversal}, the sections \,$f_t$\,
have positive semi-definite second fundamental form. Hence,
Alexander and do Carmo\,--Warner Theorems apply  and give that, for all \,$t\in\R,$\,
\,$M_t$\, is  homeomorphic to \,$\s^{n-1},$\, \,$f_t$\, is an embedding,
and \,$f(M_t)$\, bounds a convex set in \,$\overbar M_t$\,. Therefore,
\,$f$\, is  an embedding and \,$M$\, is homeomorphic to the product
\,$\s^{n-1}\times\R.$

Suppose that \,$\overbar M=\ha^n.$\, In this case, since we are assuming
the second fundamental form of \,$f$\, positive semi-definite, we can apply
Bishop Theorem as in the proof of Theorem \ref{th-th1}, and equally conclude that
the mean convex side \,$\Lambda\subset B\times\R$\, of \,$f$\, is convex (we remark that
Bishop Theorem only requires semi-definiteness of the second fundamental form). Hence,
given  \,$t\in\R,$ if we  choose
\,$x_0\in M_t$\, and a divergent sequence \,$(x_k)$\, in \,$M$\, such that
\,$\xi (x_k)\rightarrow\pm\infty,$\, each geodesic segment \,$\gamma_k$\,
of \,$\ha^n\times\R$\, from
\,$f(x_0)$\, to \,$f(x_k)$\, is contained in \,${\rm cl}\,\Lambda.$\, Consequently,
the limit geodesic ray \,$\gamma=\lim\gamma_k$\, emanating from \,$f(x_0)$\,
is contained in \,${\rm cl}\,\Lambda$\, as well. If \,$\gamma$\, were not a vertical geodesic ray, it would
eventually reach the boundary of \,$B\times\R,$\, which is impossible, since \,${\rm cl}\,\Lambda\subset B\times\R.$\,
So, \,$\gamma$\, is  vertical and tangent to \,$f$\, at \,$x_0$\,, that is,
\,$T_{x_0}M$\, is vertical. Being both
\,$t$\, and \,$x_0$\,  arbitrary, it follows that \,$f(M)=f(M_0)\times\R.$

Let us assume now that \,$\overbar M=\s^n.$\, Under this assumption,
the sectional curvature \,$K$\, of \,$M$\, is nonnegative,  for
\,$\s^n\times\R$\, has nonnegative sectional curvature and, by the hypothesis,
\,$f$\, has nonnegative extrinsic curvature.
Since \,$M$\, is noncompact, it implies that,
for all \,$x\in M,$\, there exist orthonormal vectors \,$X, Y\in T_xM$\, satisfying \,$K(X,Y)=0$\, (otherwise,
by the  Soul Theorem, \,$M$\, would be homeomorphic to \,$\R^n$).
However, from implication \eqref{eq-det} and equality \eqref{eq-daniel},
\[
\|\pi_{\scriptscriptstyle{XY}}\,\g\xi\|^2=1+\det A_{\scriptscriptstyle{XY}}\ge 1.
\]
Thus,  \,$\|\pi_{\scriptscriptstyle{XY}}\,\g\xi(x)\|=1,$\, i.e.,
\,$\|\g\xi(x)\|=1,$\, and so
\[\g\xi(x)=\partial_t \,\,\, \forall x\in M,\]
which clearly implies that \,$f(M)=f(M_0)\times\R.$\,
This finishes the proof.
\end{proof}

\begin{remark} \label{rem-neessary}
The assumption (b) in Theorem \ref{th-nocriticalpoints} is necessary.
Indeed, as noted in  \cite[pg. 124]{spivak},
there exists a proper unbounded immersion \,$g:M^{n-1}\rightarrow\h^n$\, ($n\ge 3$)
with positive definite second fundamental form, which is not an
embedding. Therefore, on one hand, the immersion \,$f:=g\times{\rm id}\,:M^{n-1}\times\R\rightarrow\h^n\times\R$\,
is proper, has positive semi-definite second fundamental form, and its height function has no critical points.
On the other hand, \,$f$\,  is not an embedding. (Notice that \,$f$\,  is non cylindrically bounded, for \,$g$\, is unbounded.)
\end{remark}

\section{Results on Hypersurfaces in \,$\R\times_{\rho}\ha^n$\, and \,$\R\times_\rho\s^n$} \label{sec-th2}

Given an $n$-dimensional  Riemannian manifold \,$\overbar M^n,$
and a positive differentiable function \,$\varrho:\R\rightarrow\R,$\,
the \emph{warped product} \,$\R\times_\varrho\overbar{M}^n$ is, by definition,
the manifold \,$\R\times\overbar M^n$ endowed
with the metric
\[
\langle X,Y\rangle=\langle X_h,Y_h\rangle_{\R} + \varrho^2(t)\langle X_v,Y_v\rangle_{\overbar{M}},
\,\, X, Y\in T_{(t,p)}(\R\times\overbar{M}), \,\, (t,p)\in \R\times\overbar{M}.
\]
Here \,$\langle \,,\, \rangle_\R$\, and \,$\langle \,,\, \rangle_{\overbar{M}}$\, denote the Riemannian metrics of  \,$\R$\,
and \,$\overbar M^n,$\, respectively,
and the notation is as in Section \ref{sec-preliminaries} (notice that, now, horizontal vectors are tangent to \,$\R,$\, whereas vertical vectors are
tangent to \,$\overbar M$).

It is easily seen that, in   \,$\R\times_\varrho\overbar M^n,$  all \emph{vertical sections}
\[\overbar{M}_t:=\{t\}\times_\varrho\overbar M^n\]
are  homothetic  to \,$\overbar M.$\,  
In particular, the following hold:

\begin{itemize}[parsep=1ex]

\item 
Vertical sections are Hadamard manifolds (resp. spheres with constant sectional curvature)
if \,$\overbar M$\, is a Hadamard manifold (resp. \,$\s^n$).

\item 
The Riemannian connection of any vertical section, to be denoted by \,$\overbar{\nabla},$\,
can be identified with that of \,$\overbar M$\, (see \cite[Lema 64, pg. 92]{oneill}).
\end{itemize}

Denoting by  \,$\widetilde\nabla$\, the Riemannian connection of \,$\R\times_\rho\overbar M^n$\, and defining
\[
\phi(t):=\frac{\rho'(t)}{\rho(t)}\,, \,\,\, t\in\R,
\]
for any  \emph{vertical} fields \,$X, Y\in T\overbar{M},$\,
the following identities hold
(see \cite[Lema 7.3]{bishop-oneill}):
\begin{equation}\label{eq-connectionwarped}
  \begin{aligned}
    \widetilde\nabla_XY &= \overbar\nabla_XY-\phi\langle X,Y\rangle\partial_t\,.\\
    \widetilde{\nabla}_X\partial_t &=\widetilde{\nabla}_{\partial_t}X=\phi X.\\
    \widetilde{\nabla}_{\partial_t}\partial_t &= 0.
  \end{aligned}
\end{equation}

Let us introduce now the concept of $\phi$-convexity of hypersurfaces in warped products \,$\R\times_\rho\overbar{M}.$\,
As we pointed out in the introduction, when \,$\rho$\, is constant, $\phi$-convexity
is equivalent to positive semi-definiteness of the second fundamental form.

\begin{definition} \label{def-phiconvex}
An oriented hypersurface \,$f:M^n\rightarrow\R\times_\rho\overbar{M}$\, is called
$\phi$-\emph{convex} (resp. \emph{strictly} $\phi$-\emph{convex}) if, for all \,$x\in M,$\,
each eigenvalue \,$\lambda$\, of its shape operator at \,$x$\,  satisfies
\,$\lambda\ge|\phi\circ\xi(x)|$\, (resp. \,$\lambda>|\phi\circ\xi(x)|$).
\end{definition}

Vertical sections are trivial examples of $\phi$-convex hypersurfaces of \,$\R\times_\rho\overbar{M}^n.$\,
In the case where \,$\overbar M$\, is a simply connected space form,
they are the only $\phi$-convex hypersurfaces
under certain restrictions  on the warp function \,$\rho$\, and the immersed manifold \,$M$\,
(see Proposition \ref{prop-warpverticalsection} at the end of this section).
In hyperbolic space \,$\h^{n+1}=\R\times_{e^{t}}\R^n,$\, as we have discussed,
the vertical sections are the constant mean curvature one horospheres. In particular, the geodesic spheres of \,$\h^{n+1}$
are all $\phi$-convex, since they are totally umbilical and have constant mean curvature greater than one.

Concerning oriented hypersurfaces \,$f:M^n\rightarrow\R\times_\rho\overbar{M}^n,$\,
we shall keep the notation of the previous sections.
Namely, for such an \,$f,$ \,$N$\, will denote its unit normal field,  \,$A$\, its shape operator,
\,$\theta=\langle N,\partial_t\rangle$\, its angle function, and \,$\xi=\pi_{\scriptscriptstyle\R}\circ f$\,
its height function. In particular, as before,
\,$\g\xi=\partial_t-\theta N.$

\begin{remark}
To avoid excessive notation, we will  write \,$\rho$\, and \,$\phi$\, for the compositions
\,$\rho\circ\xi$\, and \,$\phi\circ\xi,$\, respectively, since there is no danger of confusion.
\end{remark}

From the equalities \eqref{eq-connectionwarped}, for all \,$X\in T(\R\times_\rho\overbar{M}),$\, one  has
\begin{equation} \label{eq-connectionwarp2}
\widetilde\nabla_X\partial_t=\widetilde\nabla_{X_v}\partial_t=\phi X_v=\phi\left(X-\langle X,\partial_t\rangle\partial_t\right).
\end{equation}
Thus, if \,$X\in TM,$\,
\[
X(\theta)=\langle\widetilde\nabla_XN,\partial_t\rangle+\langle N,\widetilde\nabla_X\partial_t\rangle=-\langle A\,\g\xi,X\rangle-\phi\theta\langle\g\xi,X\rangle.
\]
Hence, the gradient of \,$\theta$\, is
\begin{equation} \label{eq-gradthetawarp}
\g\theta=-(A+\phi\theta\,{\rm Id})\,\g\xi,
\end{equation}
where \,${\rm Id}$\, stands for the identity map of \,$TM.$

Given \,$X,Y\in TM,$\, we have that
\begin{eqnarray}
{\rm Hess}\,\xi(X,Y) & = & \langle\widetilde\nabla_{X}\,\g\xi,Y\rangle=\langle\widetilde\nabla_{X_v}\partial_t,Y\rangle-\theta\langle\widetilde\nabla_{X}N,Y\rangle\nonumber\\
          & = & \phi\langle X_v\,, Y\rangle + \theta\langle\alpha(X,Y),N\rangle\nonumber\\
          & = &  \phi(\langle X,Y\rangle- \langle X,\partial_t\rangle \langle Y,\partial_t\rangle)+\theta\langle\alpha(X,Y),N\rangle. \nonumber
\end{eqnarray}
In particular,
\begin{equation}  \label{eq-laplacianwarped}
{\rm Hess}\,\xi(X,X)=\phi(\langle X,X\rangle-\langle X,\partial_t\rangle^2)+\theta\langle\alpha(X,X),N\rangle \,\,\, \forall X\in TM.
\end{equation}

We also call attention to the fact that, defining the L-operator
\[{\rm L}(\rho):=(\rho')^2 -\rho\rho'',\]
we have from  Gauss equation for hypersurfaces \,$f:M^n\rightarrow\R\times_{\rho}\q^n_\epsilon$\,
(see, e.g.,  \cite[Proposition 3]{lawn-ortega})
that, for all orthonormal tangent fields \,$X, Y\in TM,$\, the sectional curvature \,$K$\, of \,$M$\,  satisfies
\begin{equation}  \label{eq-gausswarp}
K(X,Y)=\left(\frac{\epsilon}{\rho^2}-\phi^2\right)+
\left(\frac{{\rm L}(\rho)-\epsilon}{\rho^2}\right)\|\pi_{\scriptscriptstyle{XY}}\,\g\xi\|^2+
\det A_{\scriptscriptstyle{XY}}.
\end{equation}

We proceed now to the proofs of Theorems \ref{th-warp} and \ref{th-warpnocriticalpoints}. First, we establish the following Lemma, which
can be considered as a ``warped'' version of Lemma \ref{lem-transversal}.

\begin{lemma} \label{lem-warptransversal}
Let  \,$f:M^n\rightarrow\R\times_\rho\overbar{M}^n$\, ($n\ge 3$) be a
$\phi$-convex (resp. strictly $\phi$-convex) hypersurface.
Then, for all \,$X\in TM,$\, one has
\begin{equation}\label{eq-warpsff}
\langle\alpha(X,X),N\rangle\pm\phi\theta\langle X,X\rangle\ge 0 \,\, \, ({\rm resp.}\,\,>0).
\end{equation}
Consequently, any vertical section
$f_t:M_t\rightarrow\overbar M_t$ of  \,$f$, if properly oriented,
has positive semi-definite (resp. definite) second fundamental form.
\end{lemma}

\begin{proof}
Consider an orthonormal frame \,$\{X_1\,,\dots ,X_n\}\subset TM$\, of eigenvectors
of the shape operator \,$A$\, of \,$f$\, with corresponding eigenvalues
\,$\lambda_1\,, \dots ,\lambda_n$\,.  The $\phi$\, convexity of \,$f$\, yields \,$\lambda_i\ge|\phi|$\,
for all  \,$i=1,\dots, n.$\, Since
\,$-1\le\theta\le 1,$\, by setting  \,$X=a_1X_1+\cdots+a_nX_n$\,, one has
\[
\langle\alpha(X,X),N\rangle=\langle AX,X\rangle=\sum_{i=1}^{n}\lambda_ia_i^2\ge|\phi|\langle X,X\rangle\ge\pm\phi\theta\langle X,X\rangle,
\]
which gives \eqref{eq-warpsff}.

Let us consider now a vertical section \,$f_t:M_t\rightarrow\overbar{M}_t$\, with orientation
\begin{equation}\label{eq-etawarp}
\eta=\frac{N_v}{\|N_v\|}=\frac{1}{\sqrt{1-\theta^2}}(N-\theta\partial_t).
\end{equation}
In this case, for all \,$X\in TM_t\,,$\,
\begin{eqnarray}
\langle\alpha_{f_t}(X,X),\eta\rangle & = & \frac{1}{\sqrt{1-\theta^2}}\langle\overbar\nabla_XX,N-\theta\partial_t\rangle =
\frac{1}{\sqrt{1-\theta^2}}\langle\overbar\nabla_XX,N\rangle \nonumber\\
                                     & =  &\frac{1}{\sqrt{1-\theta^2}}(\langle\widetilde\nabla_XX,N\rangle+\phi\theta\langle X,X\rangle), \nonumber
\end{eqnarray}
where the last equality followed from the first one in   \eqref{eq-connectionwarped}.
This, together with inequality \eqref{eq-warpsff}, gives
\[
\langle\alpha_{f_t}(X,X),\eta\rangle= \frac{1}{\sqrt{1-\theta^2}}(\langle\alpha(X,X),N\rangle+\phi\theta\langle X,X\rangle)\ge 0,
\]
which implies that \,$f_t$\, has positive semi-definite second fundamental form.
If \,$f$\, is strictly $\phi$-convex, the inequality
\eqref{eq-warpsff} is strict and, then, \,$f_t$\, has positive
definite second fundamental form.
\end{proof}

\begin{proof}[Proof of Theorem \ref{th-warp}]
Let \,$x_0\in M$\, be a critical point of the height function
\,$\xi$\, of \,$f.$\,
It follows from the Hessian formula \eqref{eq-laplacianwarped} and inequality
\eqref{eq-warpsff} that  \,$x_0$\, is a strict maximum   if
\,$\theta(x_0)=-1,$\, or is a strict minimum if \,$\theta(x_0)=1.$\, Let us assume
the latter, and also that \,$\xi(x_0)=0.$

Define normal sections, normal regions and normal values for \,$x_0$\, as in the proof of Theorem \ref{th-th1}.
As before, denote by  \,$\Omega\subset M$\, the union of all normal regions,  and by \,$I\subset (0,+\infty)$\, the interval
of all normal values. Recall that \,$\Omega$ is
homeomorphic to \,$\R^n$\, and consider the cases:

\begin{itemize}[parsep=1ex]
  \item[i)] $\Omega=M.$
  \item [ii)]$\Omega\ne M$\, and \,$\partial\Omega$\, contains critical points of \,$\xi.$
  \item [iii)] $\Omega\ne M$\, and \,$\partial\Omega$\, contains no critical points of \,$\xi.$
\end{itemize}

The vertical sections \,$\overbar M_t$\, are either all Hadamard manifolds
or all spheres with constant sectional curvature.
Thus,  do Carmo\,--Warner and Alexander Theorems, together with Lemma \ref{lem-warptransversal},
imply that each normal section
\,$f_t:M_t\rightarrow\overbar{M}_t$\, is an embedding.
Therefore, if (i) occurs, \,$M$\, is homeomorphic to \,$\R^n$\, and \,$f$\, is a proper embedding.
In particular, if \,$\overbar M=\s^n,$\,  the height function of
\,$f$\, is unbounded,  since the region bounded by two normal sections
of  \,$\R\times_\rho\s^n$\, is compact.

If (ii) holds, then we can argue just as in the proof of Theorem \ref{th-th1}
to conclude that \,$\xi$\, has precisely two critical points, giving that
\,$M$\,  is homeomorphic to \,$\s^n.$\,

Finally, let us assume that (iii) holds and then derive a contradiction.
Reasoning as in the proof of Theorem \ref{th-th1}, it suffices to prove that
\,$M_{t^*}:=\partial\Omega$\, is compact, where \,$t^*=\sup I.$\,

As before, we have
that \,$f_{t^*}:M_{t^*}\rightarrow\overbar M_{t^*}$\, is a proper embedding.
If \,$\overbar M^n=\s^n,$\,
then  \,$\overbar M_{t^*}$\, is  a sphere, which implies that \,$M_{t^*}$\, is compact, since  \,$f_{t^*}$\, is proper.
Hence, (iii) does not hold if \,$\overbar M^n=\s^n.$

Let us suppose now  that \,$\overbar M^n$\, is a Hadamard manifold.
Under this hypothesis, we shall  prove that the  projections 
of  \,$f(\Omega_t)$\, to \,$\overbar M_{t^*}$\, are uniformly bounded, which will imply that
\,$M_{t^*}$\, is compact. In our reasoning, we will use some ideas contained in Currier's proof of his \cite[Theorem A]{currier}.

Let \,$\pi:\R\times_\rho\overbar M^n\rightarrow\overbar M_{t^*}$\,
be the  projection onto \,$\overbar M_{t^*}.$\, Given \,$t$\, in \,$(0,t^*),$\,
set \,$\Omega_t^*=\pi(f(\Omega_t))\subset\overbar M_{t^*},$\,
consider a boundary point \,$x^*\in\partial\Omega_t^*,$\, and let \,$x\in {\rm cl}\,\Omega_t$\, be such that
\,$\pi(f(x))=x^*.$\,  Then, choose \,$t_0\in (0,t^*),$\, \,$t_0<t,$\, in such a way that
the angle function \,$\theta$\, is positive in \,${\rm cl}\,\Omega_{t_0}$\, (notice that \,$\theta(x_0)=1$).
In this setting, we have either
\,$x\in\partial\Omega_t=M_t$\, or \,$x\in\Omega_t\,.$\,
In the latter case, it is clear that the horizontal  geodesic (i.e., parallel to \,$\partial_t$)
through \,$f(x)$\, is tangent to \,$M,$\, that is, \,$\theta(x)=0.$\,
So, in any case, \,$x\not\in{\rm cl}\,\Omega_{t_0}$\,  (Fig. \ref{fig-projection}).

\begin{figure}[htbp]
\includegraphics{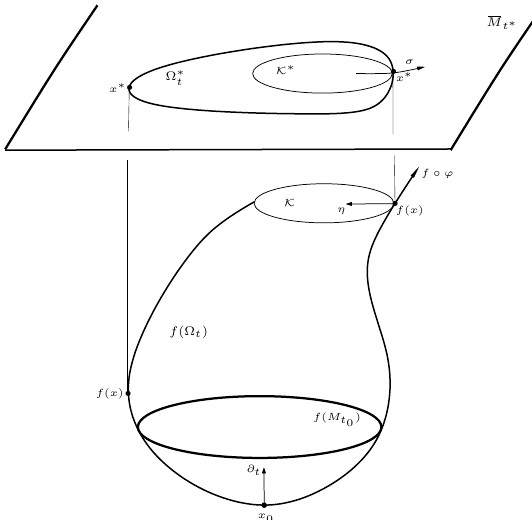}
\caption{Projection of $f(\Omega_t)$ to $\overbar M_{t^*}$}
\label{fig-projection}
\end{figure}

Let \,$\varphi:[0,T+\delta]\rightarrow M,$\, \,$\delta >0,$\, be an arclength  parametrization of the trajectory of
\,$\g\xi$\, through \,$x$\, satisfying \,$\varphi(0)\in M_{t_0}$\, and \,$\varphi(T)=x.$\,
Write
\[\varphi'(s)=a(s)\eta(\varphi(s))+b(s)\partial_t, \,\, s\in [0,T+\delta],\]
where  \,$\eta$\, is defined as in \eqref{eq-etawarp}. We remark
that \,$b$\, is a positive function, since \,$\varphi'$\, is parallel to \,$\g\xi.$\,
Also, it is easily seen that, along \,$\varphi,$\, \,$N=b\eta-a\partial_t.$\,
So,
\[a(s)=-\theta(\varphi(s)), \, s\in [0,T+\delta].\]
From this,  equality \eqref{eq-gradthetawarp}, and inequality \eqref{eq-warpsff}, we have
\begin{equation}\label{eq-thetawarpdecreasing}
a'=\langle A\,\g\xi, \varphi'\rangle+\phi\theta\langle\g\xi,\varphi'\rangle=
\|\g\xi\|\left(\langle A\,\varphi', \varphi'\rangle+\phi\theta\langle\varphi',\varphi'\rangle\right)>0,
\end{equation}
which implies that \,$a$\, is increasing.

Let us show now that $a(T)\le 0.$\, Indeed, if $x=\varphi(T)\in\Omega_t$\,, as we know,
$a(T)=-\theta(x)=0.$  So, we can assume   \,$x\in\partial\Omega_t=M_t$\, and  \,$a(T)\ne 0.$\,
In this case, denoting by \,$\mathcal{K}$\, the convex set bounded by  \,$f(M_t)$\, in \,$\overbar M_t$\,,
and setting \,$\mathcal{K}^*:=\pi(\mathcal{K})\subset\overbar M_{t^*},$\, we have that \,$x^*\in\partial\Omega_t^*\cap\partial \mathcal{K}^*.$\,
Clearly, \,$\mathcal{K}^*$\, is compact,  and \,$\pi_*\eta(x)$\, is orthogonal to \,$\partial\mathcal{K}^*,$\,
pointing inward \,$\mathcal{K}^*.$\, Furthermore, \,$f({\rm cl}\,\Omega_t)$\, separates \,$\overbar M^n\times [0,t],$\,
for  \,$f|_{{\rm cl}\,\Omega_t}$\, is an embedding. Hence, a horizontal geodesic through any point in \,$\mathcal{K}^*$\,
necessarily reaches \,$f(\Omega_t)$\,, which yields  \,$\mathcal{K}^*\subset\Omega_t^*$\, (see Fig. \ref{fig-projection}).

Consider the projection \,$\sigma(s):=\pi(f(\varphi(s)))$\, of
\,$f\circ\varphi$\, to \,$\overbar M_{t^*}.$\, Since
\,$\sigma'(T)=a(T)\pi_*\eta(x)\ne 0,$\, one has that
\,$\sigma$\, is transversal to \,$\partial\mathcal{K}^*$\, (and then to \,$\partial\Omega_t^*$)
at \,$\sigma(T)=x^*.$\,
Therefore, the part of \,$\sigma$\, inside (respectively, outside) \,$\mathcal{K}^*$\, is the  projection
of a part of \,$f\circ\varphi$\, inside (respectively, outside) \,$f(\Omega_t)$\,,
that is, \,$\sigma(s)\in \mathcal{K}^*$\, for all small \,$s<T$\, and
\,$\sigma(s)\not\in \mathcal{K}^*$\, for \,$s>T.$\,
This gives that the velocity vector \,$\sigma'(T)$\, points outward \,$\mathcal{K}^*$\, at \,$\sigma(T)=x^*,$\,
that is, \,$0>\langle\sigma'(T),\pi_*\eta(x)\rangle=a(T),$\, as claimed.

Since \,$a$\, is increasing and both \,$a(0)$\, and \,$a(T)$\, are non positive, we have that
\,$a\le 0$\,  in \,$[0,T].$\, But \,$a^2+b^2=1.$\, Hence, \,$aa'+bb'=0,$\, which implies
that \,$b'\ge 0$\, in \,$[0,T],$\, for \,$b>0.$\,
Therefore, \,$b$\, is nondecreasing. So, if we set
\[
\lambda:=\inf_{M_{t_0}}\left\langle\frac{\g\xi}{\|\g\xi\|}, \partial_t\right\rangle,
\]
we have that  \,$b=\langle\varphi',\partial_t\rangle\ge\lambda>0.$\,
Thus,
\[
t-t_0=\xi(\varphi(T))-\xi(\varphi(0))=\int_0^T(\xi\circ\varphi)'(s)ds=\int_0^T\langle\g\xi,\varphi'\rangle ds=
\int_0^Tb(s)ds\ge T\lambda,
\]
and so, the following inequalities hold:
\begin{equation}\label{eq-T}
T\le\frac{t-t_0}{\lambda}<\frac{t^*-t_0}{\lambda}\,\cdot
\end{equation}

Now, set \,$t(s):=\xi(\varphi(s))\,, \, s\in [0,T+\delta],$\, and notice that
\[
1=\|\eta(\varphi(s))\|=\rho(t(s))\|\eta(\varphi(s))\|_{\overbar M}.
\]
Since \,$\sigma'(s)=a(s)\pi_*\eta(\varphi(s))$\, and \,$\|\eta(\varphi(s))\|_{\overbar M}=\|\pi_*\eta(\varphi(s))\|_{\overbar M},$\, we  have
\[
\|\sigma'(s)\|=\rho(t^*)|a(s)|\|\eta(\varphi(s))\|_{\overbar M}=\frac{\rho(t^*)}{\rho(t(s))}|a(s)|\le\frac{\rho(t^*)}{\mu}|a(s)|,
\]
where \,$\mu=\inf\rho|_{[t_0,t^*]}.$\,

Therefore, denoting the  length  of \,$\sigma$\, from
\,$0$\, to \,$T$\, by \,$\mathcal{L}(\sigma),$\,  considering \eqref{eq-T},
and taking into account that \,$a=\sqrt{1-b^2}\le\sqrt{1-\lambda^2},$\, one has
\[
\mathcal{L}(\sigma)=\int_0^T\|\sigma'(s)\|ds\le\frac{\rho(t^*)}{\mu}\int_0^T|a(s)|ds
<\frac{\rho(t^*)(t^*-t_0)}{\lambda\mu}\sqrt{1-\lambda^2}.
\]

It follows that each point \,$x^*\in\partial\Omega_{t}^*$\, can be joined to a point of
\,$\pi(f(\Omega_{t_0}))$\, by a curve whose length is bounded by a constant independent of  \,$t,$\,
which clearly  implies that \,$\Omega_t^*=\pi(f(\Omega_t))$\,  is uniformly bounded. Consequently, \,$\pi(f(\Omega))$\, is bounded in
\,$\overbar M_{t^*}.$\, In particular,  \,$M_{t^*}$\, is bounded, and so is  compact, as we wished to show.
This fact, as we pointed out, leads to a contradiction and
then finishes the proof of the theorem.
\end{proof}

\begin{remark} \label{rem-convex}
Under the conditions of Theorem \ref{th-warp}, consider the case when \,$\overbar M^n$\, is a Hadamard
manifold \,$\ha^n,$\, and assume further that the warp function \,$\rho$\, is convex, that is,
\,$\rho''\ge 0.$\, Then, in addition to the conclusions of the theorem, one has:
\begin{itemize}[parsep=1ex]
  \item[a)] $f(M)$\, bounds a convex set in \,$\R\times_\rho\ha^n.$
  \item[b)] When \,$M$\, is homeomorphic to \,$\R^n,$
  \,$f(M)$\, is an unbounded horizontal geodesic graph over an open set of a vertical section of
  \,$\R\times_\rho\ha^n$.
\end{itemize}
Indeed, by \cite[Theorem 7.5]{bishop-oneill} (see also \cite[Remark 7.7]{bishop-oneill}), the convexity of
\,$\rho$\, implies that
\,$\R\times_\rho\mathscr{H}^n$\, is a Hadamard manifold. This fact, as can be easily seen,  allows us to
mimic the first part of the proof of Theorem \ref{th-th1} and, then, get (a) and (b).
\end{remark}

\begin{proof}[Proof of Theorem \ref{th-warpnocriticalpoints}]
As in the proof of Theorem \ref{th-nocriticalpoints}, it follows from the compacity of the vertical sections
of \,$\R\times_{\rho}\s^n,$\, the properness of \,$f,$\, and the absence of critical points of  its height function that
there exists a family of  vertical sections,
\[f_t:M_t\subset M\rightarrow S_t:=\{t\}\times_{\rho}\s^n, \, t\in\R,\]
such that \,$M=\bigcup M_t$\,, and \,$M_t$\, is compact and connected for all \,$t\in\R.$

Also, by Lemma \ref{lem-warptransversal} and
do Carmo\,--Warner Theorem, for all \,$t\in\R,$\,
\,$M_t$\, is  homeomorphic to \,$\s^{n-1},$\, \,$f_t$\, is an embedding,
and \,$f(M_t)$\, bounds a convex set in \,$S_t$\,. Thus,
the hypersurface \,$f$\, itself is  an embedding
and \,$M$\, is homeomorphic to \,$\s^{n-1}\times\R.$

Let us assume now that \,$M$\, has nonnegative sectional curvature.
Since \,$M$\, is noncompact and not homeomorphic to \,$\R^n,$\,  Perelman Soul Theorem gives that,
for all \,$x\in M,$\, there exist orthonormal vectors \,$X, Y\in T_xM$\, satisfying \,$K(X,Y)=0.$\,
Then, considering  equality \eqref{eq-gausswarp}, we  have
\begin{eqnarray}
  0 &=& \left(\frac{1}{\rho^2}-\phi^2\right)+\frac{{\rm L}(\rho)-1}{\rho^2}\|\pi_{\scriptscriptstyle{XY}}\,\g\xi\|^2+\det A_{\scriptscriptstyle{XY}}\nonumber\\
    &\ge & \frac{1}{\rho^2}+\frac{{\rm L}(\rho)-1}{\rho^2}\|\pi_{\scriptscriptstyle{XY}}\,\g\xi\|^2, \nonumber
\end{eqnarray}
for \,$\det A_{\scriptscriptstyle{XY}}\ge\phi^2,$\, by the $\phi$-convexity of \,$f$\, (see \eqref{eq-det}).

Therefore,  the inequality
\begin{equation} \label{eq-inequalitywarp}
 (1-{\rm L}(\rho))\|\pi_{\scriptscriptstyle{XY}}\,\g\xi\|^2\ge 1
\end{equation}
holds  and yields  \,${\rm L}(\rho)\le 0,$\,  since \,$\|\pi_{\scriptscriptstyle{XY}}\,\g\xi\|\le 1.$\,

From \eqref{eq-inequalitywarp}, we also have that  \,$\|\g\xi\|=1$\, if \,${\rm L}(\rho)= 0$\, on \,$\R.$\,
In this case, \,$\g\xi=\partial_t$\, on all of \,$M.$\,
In particular, \,$\langle N,\partial_t\rangle=0.$\, Differentiating this equality and considering \eqref{eq-connectionwarp2},
we easily conclude that \,$A\partial_t=0.$\, This, together with the $\phi$-convexity of \,$f,$\, implies that
\,$\phi=0$\, on \,$\R$\, and, then, that \,$\rho$\, is constant.
\end{proof}

In conclusion, we point out the following property of $\phi$-convex hypersurfaces: 

\begin{proposition} \label{prop-warpverticalsection}
{If $f:M^n\rightarrow\R\times_\rho\q_\epsilon^n$ is $\phi$-convex, $M$ has  sectional curvature $K\le\epsilon/\rho^2,$
and \,${\rm L}(\rho)>\epsilon$ on \,$M,$  then each connected component of \,$f(M)$ is contained in a vertical section of \,$\R\times_\rho\q_\epsilon^n$}.
\end{proposition}
\begin{proof}
From the hypothesis on \,$K$\, and equality \eqref{eq-gausswarp}, one has
\begin{eqnarray}
  \frac{\epsilon}{\rho^2} &\ge& \left(\frac{\epsilon}{\rho^2}-\phi^2\right)+
  \frac{{\rm L}(\rho)-\epsilon}{\rho^2}\|\pi_{\scriptscriptstyle{XY}}\,\g\xi\|^2+\det A_{\scriptscriptstyle{XY}}\nonumber\\
    &\ge & \frac{\epsilon}{\rho^2}+\frac{{\rm L}(\rho)-\epsilon}{\rho^2}\|\pi_{\scriptscriptstyle{XY}}\,\g\xi\|^2, \nonumber
\end{eqnarray}
which implies that \,$\pi_{\scriptscriptstyle{XY}}\,\g\xi=0 \, \forall X, Y\in TM,$\, that is,
\,$\g\xi =0$\, on \,$M.$
\end{proof}

It is easily seen that Proposition \ref{prop-warpverticalsection}
applies to  the following type of hypersurfaces:

\begin{itemize}[parsep=1ex]
\item $f:M^n\rightarrow\R\times_{e^{-t^2/2}}\R^n$\,  with\,  $K\le 0$.
\item $f:M^n\rightarrow\R\times_{\cosh (t/2)}\h^n$\, with\, $K\le -1.$
\end{itemize}

\vt
\vt
\noindent
{\bf Acknowledgments.} We are indebted to Fernando Manfio and Ruy Tojeiro for  valuable
suggestions which improved some results in this paper. We would also like
to thank  Luis Florit for helpful conversations.

\end{document}